\documentclass[11pt]{amsart}  %used for final

\usepackage[utf8]{inputenc}
\usepackage{amsmath}
\usepackage{amsfonts}
\usepackage{amssymb}
\usepackage{latexsym}
\usepackage{graphicx}
\usepackage{subfigure}
\usepackage{color}
\usepackage{mathtools}
\usepackage[backref=page]{hyperref}
\usepackage{verbatim}
\usepackage[all]{xy}
\usepackage{pdfsync}
\usepackage{youngtab}
\usepackage{ytableau}
\usepackage{young}
\usepackage{tikz}
\usepackage{cancel} % for gigantic cancellation marks
\usepackage[normalem]{ulem} % for strike throughs
\usepackage{graphicx}
\usepackage{forest}
\usepackage{array}
\usepackage{extarrows}
\usepackage[mathscr]{euscript}
\usepackage[textsize=tiny]{todonotes}
\theoremstyle{plain}
\newtheorem{theorem}{Theorem}[section]
\newtheorem{proposition}[theorem]{Proposition}

\newtheorem{lemma}[theorem]{Lemma}
\newtheorem{corollary}[theorem]{Corollary}

\newtheorem{example}[theorem]{Example}

\newtheorem{remark}[theorem]{Remark}

\numberwithin{equation}{section}
\newcommand{\suchthat}{\:|\:} %set theoretic

\newcommand{\bC}{\mathbb{C}}

\newcommand{\bQ}{\mathbb{Q}}
\newcommand{\bR}{\mathbb{R}}
\newcommand{\bZ}{\mathbb{Z}}

\newcommand{\mc}[1]{\mathcal{#1}} %calligraphic
\definecolor{emphcol}{rgb}{0.0, 0.5, 2.0}
\newcommand{\brkd}{\mathrm{Break}} % set of break divisors
\newcommand{\orid}{\mathrm{Orientable}} % set of break divisors
\newcommand{\hilb}{\mathrm{Hilb}} % Hilbert series
\newcommand{\aut}[1]{\mathrm{Aut}(#1)} % automorphism group of a graph
\newcommand{\prim}{\mathrm{prim}} % Efimov's prim
\newcommand{\edges}{\mathscr{E}} % edge
\newcommand{\vertices}{\mathscr{V}} % vertices
\newcommand{\cuts}{\mathscr{C}}%cuts
\newcommand{\bonds}{\mathscr{B}} % set of bonds
\newcommand{\frob}{\mathrm{Frob}} % Frobenius
\newcommand{\grfrob}{\mathrm{grFrob}} % graded Frobenius
\newcommand{\bfx}[1]{\mathbf{x}_{#1}} % bold alphabet

\begin{document}

\title[Zonotopal algebras, orbit harmonics, and DT invariants]{Zonotopal algebras, orbit harmonics, and Donaldson--Thomas invariants of symmetric quivers}
\author{Markus Reineke}
\address{Faculty of Mathematics, Ruhr University Bochum, Bochum, Germany}
\email{\href{mailto:markus.reineke@rub.de}{markus.reineke@rub.de}}

\author{Brendon Rhoades}
\address{Department of Mathematics, University of California San Diego, La Jolla, CA 92093, USA}
\email{\href{mailto:bprhoades@math.ucsd.edu}{bprhoades@math.ucsd.edu}}

\author{Vasu Tewari}
\address{Department of Mathematics, University of Hawaii at Manoa, Honolulu, HI 96822, USA}
\email{\href{mailto:vvtewari@math.hawaii.edu}{vvtewari@math.hawaii.edu}}
\thanks{M.~R. acknowledges financial support from the DFG CRC-TRR 191 ``Symplectic structures in geometry, algebra and dynamics''. B.~R. was partially supported by NSF Grant DMS-1953781. V.~T. acknowledges the support from Simons Collaboration Grant \#855592. }

%\subjclass[2010]{Primary 05E05, 20C08; Secondary 05A05, 05E10, 05E15, 06A07, 16T05, 20C30}
%\keywords{}

\begin{abstract}
 We apply the method of orbit harmonics to the set of break divisors and orientable divisors on graphs to obtain the central and external zonotopal algebras respectively.
 We then relate a construction of Efimov in the context of cohomological Hall algebras to the central zonotopal algebra of a graph $G_{Q,\gamma}$ constructed from a symmetric quiver $Q$ with enough loops and a dimension vector $\gamma$.
This provides a concrete combinatorial perspective on the former work, allowing us to identify the quantum Donaldson--Thomas invariants as the Hilbert series of the space of $S_{\gamma}$-invariants of the Postnikov--Shapiro slim subgraph space attached to $G_{Q,\gamma}$.
 The connection with orbit harmonics in turn allows us to give a manifestly nonnegative combinatorial interpretation to numerical Donaldson--Thomas invariants as the number of $S_{\gamma}$-orbits under the permutation action on the set of break divisors on $G_{Q,\gamma}$.
 We conclude with several representation-theoretic consequences, whose combinatorial ramifications may be of independent interest.
\end{abstract}

\maketitle
%\tableofcontents

%%%%%%%%%%%%%%%%%%%
\section{Introduction}
\label{sec:intro}
%%%%%%%%%%%%%%%%%%%

Our main ingredient is a connected graph $G$ with vertex set $\vertices(G)$ typically identified with $[n]\coloneqq \{1,\dots,n\}$.
Attached to $G$ are two mathematical notions: its automorphism group $\aut{G}$, and the graphical matroid $M_G$ constructed from the (oriented) incidence matrix $\hat{X}_G$ with column vectors $e_i-e_j$ where $\{i<j\}$ is an edge of $G$.
It is classical that the bases of $M_G$ are in bijection with edge sets of spanning trees of $G$.

At the heart of this article are several spaces built using $\hat{X}_G$ that carry $\aut{G}$-representations which are non-obviously isomorphic. We link notions and ideas with genesis in approximation theory (work of Dahmen--Micchelli on splines \cite{DM85}), graph theory and tropical geometry (sandpile model and divisors on graphs, after \cite{ABKS14,BW18,PS03}), orbit-harmonics (as outlined by Kostant \cite{Kostant} and subsequently employed to great effect by the Macdonald polynomials community taking the cue from Garsia--Procesi \cite{GP92}), and Cohomological Hall algebras (work of Efimov \cite{Efi}, Kontsevich--Soibelman \cite{KS11}, Mozgovoy \cite{Mo13}, and Reineke \cite{Rei12} among others).
Figure~\ref{fig:big_picture} describes the main protagonists in this article and how they link. We proceed to discuss our results.

\begin{figure}
\includegraphics[scale=0.65]{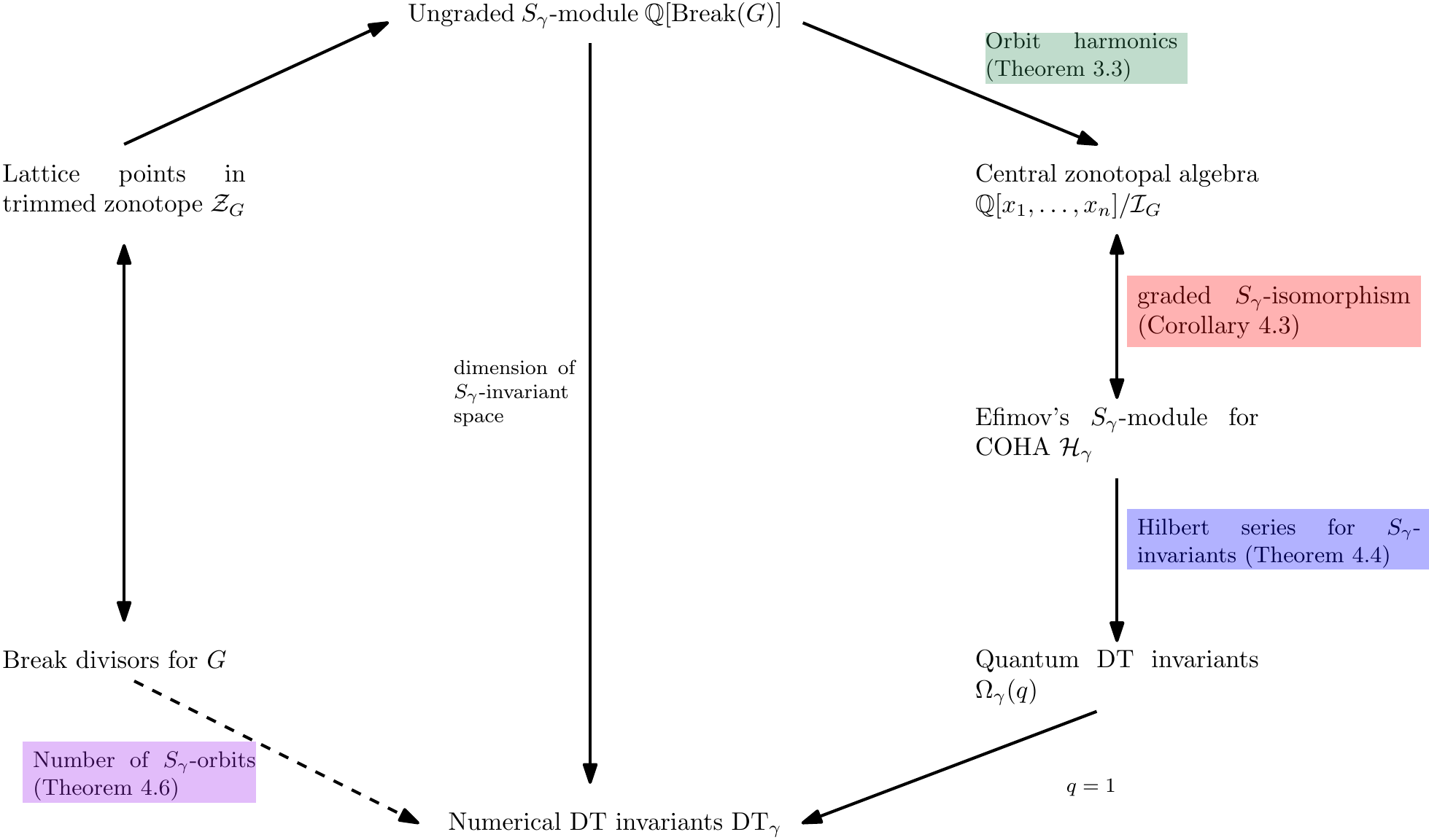}
\caption{Summary of article}
\label{fig:big_picture}
\end{figure}

%%%%%%%%%%%%%%%%%%%%%%%%%%%%%%%%%%%%%
\subsection{Discussion of main results}
\label{subsec:discussion}
%%%%%%%%%%%%%%%%%%%%%%%%%%%%%%%%%%%

The first space is the \emph{central $P$-space} $\mc{P}(G)\coloneqq \mc{P}(\hat{X}_G)$. In general, the central $P$-space is attached to a collection $X$ of nonzero vectors and forms one half of a pair of spaces $(\mc{P}(X),\mc{D}(X))$.
These spaces were introduced and studied intensively  by the approximation theory community in the 1980s in the context of splines; see \cite{dBHR93} for a book-length treatment, and to \cite[Section 1.2]{HR11} for brief yet insightful historical context. Our primary reference is the fantastic text of De Concini--Procesi \cite{DCP} that gives an in-depth survey of the variety of ways that the study of splines touches seemingly disparate areas in mathematics.
The space $\mc{D}(X)$ is called the \emph{Dahmen--Michelli space} and is the vector space dual to $\mc{P}(X)$.
The space $\mc{P}(X)$ was introduced later \cite{AA88, DR90} and turns out to be easier to work with. As we demonstrate in this article, the space
$\mc{P}(G)$ is hiding in Efimov's work \cite{Efi} resolving a conjecture of Kontsevich--Soibelman \cite{KS11}.

We denote the polynomial ring $\bQ[x_1,\dots,x_n]$ in $n$ variables $x_1$ through $x_n$ by $\bQ[\bfx{n}]$.
Our second space is a quotient $\bQ[\bfx{n}]/\mc{I}(G)$.
Here $\mc{I}(G)$ is a homogeneous ideal generated by powers of certain linear forms determined from $G$.
These ideals are instances of \emph{power ideals}, whose general theory was developed in work of Ardila--Postnikov \cite{AP10} beautifully bringing results obtained and phenomena observed by various authors (for instance \cite{Be10,DM85,HR11,OT94,PS06, SX10, Te02}) under the same umbrella.
Like $\mc{P}(X)$, there is a power ideal $\mc{I}(X)$ defined for collections of nonzero vectors $X$.
In fact, $\mc{P}(X)$ is the Macaulay-inverse to the ideal $\mc{I}(X)$, and this will be pertinent for our purposes.
The quotient $\bQ[\bfx{n}]/\mc{I}(G)$ has been called the \emph{central zonotopal algebra} by Holtz--Ron \cite[Section 3]{HR11}.

It is known \cite{PS03} that the central zonotopal algebra has a monomial basis indexed by \emph{$G$-parking functions}, which are essentially the same as \emph{$q$-reduced divisors} \cite{BN07,BS13}.
These combinatorial objects are in bijection with the set of spanning trees of $G$.
Closely related to these divisors are the \emph{\textup{(}integral\textup{)} break divisors}, and these are the ones relevant to our purposes.
Our point of departure is the fact that spanning trees on $G$ are also in bijection with break divisors on $G$ \cite{ABKS14,MZ08}; see \cite{BBY19,Yu17} for an in-depth study of \emph{geometric} bijections that arise in this framework.

A first hint that the set of break divisors $\brkd(G)$ holds interesting representation-theoretic phenomena comes from considering the case where $G$ is the complete graph $K_n$.
Then $|\brkd(K_n)|=n^{n-2}$, and the fact that $\aut{K_n}$ is the symmetric group $S_n$ translates to $\brkd(K_n)$ carrying an ungraded $S_n$-representation. Given the numerology, one naturally suspects a link to the classical parking function representation $\mathrm{PF}_{n-1}$ \cite{Hai94} of $S_{n-1}$.
Indeed, as shown in \cite{KST21}, and more generally in \cite{KRT21}, the $S_n$ action on $\brkd(K_n)$ when restricted to $S_{n-1}$ is isomorphic to $\mathrm{PF}_{n-1}$; in other words, the $S_n$-representation $\brkd(K_n)$ `extends the parking space'. This allows us establish a link with one of main motivations for this article\textemdash{}
Berget--Rhoades had already proposed a (graded) extension to $\mathrm{PF}_{n-1}$ using the \emph{slim subgraph space} of Postnikov--Shapiro \cite{PS03}.
Konvalinka--Tewari \cite[Conjecture 3.3]{KT21} conjectured that $\brkd(K_n)$ as an ungraded representation is isomorphic to the Berget--Rhoades extension. 
In particular, this would imply that the ungraded Frobenius characteristic of the Berget--Rhoades extension is $h$-positive, which is not at all immediate.
In view of these results, it is natural to inquire about the case of more general graphs.

Surprisingly, the method of \emph{orbit harmonics} allows us to tie these various strands together. Briefly put, one first interprets divisors in $\brkd(G)$ as lattice points in $\bQ^n$. One can attach a quotient $\bQ[\bfx{n}]/\mathsf{T}(G)$, where the ideal $\mathsf{T}(G)$ is homogenous, carrying a graded
$\aut{G}$-representation.
We have the isomorphism $\bQ[\brkd(G)] \cong \bQ[\bfx{n}]/\mathsf{T}(G)$ of ungraded $\aut{G}$-modules, so that
$\bQ[\bfx{n}]/\mathsf{T}(G)$ gives a graded refinement of the permutation module $\bQ[\brkd(G)]$.
 This ideal $\mathsf{T}(G)$ surprisingly turns out to equal the power ideal $\mc{I}(G)$ defining the central zonotopal algebra earlier.

Here is our first main result.
\begin{theorem}[restatement of Theorem~\ref{thm:piecing_things}]
\label{thm:main_1}
We have the following isomorphisms and equalities of ungraded $\aut{G}$-modules
\[
\bQ[\brkd(G)]\cong \bQ[\bfx{n}]/\mathsf{T}(G) = \bQ[\bfx{n}]/\mc{I}(G) \cong \mc{P}(G),
\]
where the middle equality and right isomorphism are in the category of graded $\aut{G}$-modules.
\end{theorem}

In a completely different direction, the case $G=K_n$ again provides some tantalizing hints by way of the following fact.
The number of orbits under the $S_n$-action on $\brkd(K_n)$  equals the numerical Donaldson--Thomas (DT) invariant $\mathrm{DT}_{Q,(n)}$ where $Q$ is the $2$-loop quiver and the dimension vector is $(n)$; see \cite[Theorem 3.7]{KRT21}.
\emph{What is the underlying reason for why this is the case, or is this coincidence purely numerological?}

As we now describe, this story permits a generalization to numerical DT invariants of symmetric quivers $Q$.
We work under the assumption that $Q$ has at least one loop at each vertex, i.e. \emph{has enough loops}. Given a dimension vector $\gamma$, we associate a connected graph $G_{Q,\gamma}$. All our spaces mentioned before make sense for $G_{Q,\gamma}$, even if it is unclear what they are good for.

Efimov \cite[Section 3]{Efi} describes an algebraic construction for a vector space $V$ that freely generates the cohomological Hall algebra $\mc{H}$ of a symmetric quiver $Q$. The space $V$ is built out of various graded spaces $V_{\gamma}^{\prim}$ as $\gamma$ varies over all dimension vectors for $Q$. For a fixed $\gamma$, the dimensions of the graded pieces $V_{\gamma,k}^{\prim}$ give the coefficients of the \emph{quantum} DT invariant $\tilde{\Omega}_{\gamma}(q)$, and the sum of these coefficients, i.e. the evaluation of $\tilde{\Omega}_{\gamma}$ at $q=1$, is the numerical DT invariant $\mathrm{DT}_{Q,\gamma}$.
Each space $V_{\gamma}^{\prim}$ is itself the space of $S_{\gamma}\coloneqq S_{\gamma_1}\times \cdots \times S_{\gamma_k}$ invariants of a larger space that we call $W_{\gamma}$.
We emphasize here that this larger space plays no role in \cite{Efi}.

It turns out that `combinatorializing' Efimov's construction using the graph $G_{Q,\gamma}$ results in exactly studying the space of $S_{\gamma}$-invariants of slim subgraph space $\mc{P}(G_{Q,\gamma})$. In fact, $G_{Q,\gamma}$ is such that $S_{\gamma}$ is a subgroup of $\aut{G_{Q,\gamma}}$. This allows us to put the chain of isomorphisms in Theorem~\ref{thm:main_1} to good effect and obtain:

%We summarize the discussion above tersely as our second main result. The reader is welcome to see Section~\ref{sec:Efimov} for an in-depth discussion.
\begin{theorem}[restatement of Theorems~\ref{thm:graded_multiplicity_quantum_DT} and~\ref{thm:numerical_dt_break}]
\label{thm:main_2}
The dimensions of the graded pieces of $\mathcal{P}(G_{Q,\gamma})^{S_{\gamma}}$ are the coefficients of the quantum DT invariant $\tilde{\Omega}_{\gamma}(q)$. In particular, the numerical DT invariant $\mathrm{DT}_{Q,\gamma}\coloneqq \tilde{\Omega}_{\gamma}(1)$ equals the number of $S_{\gamma}$-orbits on $\brkd(G)$.
\end{theorem}
Theorem~\ref{thm:main_2} gives, to the best of our knowledge, the first  manifestly nonnegative combinatorial interpretation for $\mathrm{DT}_{Q,\gamma}$.
It bears emphasizing that this combinatorial interpretation relies crucially on employing orbit harmonics to the point set determined by break divisors, and this is the novelty of our approach.

An algebraic interpretation for $\tilde{\Omega}_{\gamma}(q)$ distinct from that in Theorem~\ref{thm:main_2} was given in the groundbreaking work of Hausel--Letellier--Rodriguez-Villegas \cite{HLRV13} settling a conjecture of Kac \cite{Ka83}.
Recently two other interesting interpretations for $\tilde{\Omega}_{\gamma}(q)$ have been given \cite{DFR21,DM21}, both of which identify the coefficients as the dimensions of some graded piece in an algebra. The objects involved are vertex algebras, and deriving our combinatorial interpretation from these interpretations does not appear to be straightforward.
It would be an interesting endeavor relating these approaches to ours.

\smallskip

\noindent \textbf{Outline of the article.}
Section~\ref{sec:central_zonotopal} introduces the central $P$-space as well as the central zonotopal algebra. Sections~\ref{subsec:graphic_matroid} and \ref{subsec:variations} discuss the only vector configurations that concern us. In this setting, the central $P$-space is the slim subgraph space of Postnikov--Shapiro.
In Section~\ref{sec:orbit_harmonics}, we recall the notion of break divisors, recount the method of orbit harmonics, and assemble the pieces to prove our first main result (Theorem~\ref{thm:main_1} above).
We digress briefly in Section~\ref{subsec:more_orbit_harmonics} and realize internal and external zonotopal algebras by applying orbit harmonics to orientable divisors.
We describe Efimov's construction in Section~\ref{subsec:quantum_dt}, and in Section~\ref{subsec:symmetric_matrix_to_graph} introduce the covering graph $G_{Q,\gamma}$ whose slim subgraph space encodes quantum DT-invariants, thereby establishing the first half of Theorem~\ref{thm:main_2}. The second half is established in Section~\ref{subsec:numerical_DT}. 
%We link Efimov's construction to that of Hausel--Sturmfel's as well.
Section~\ref{sec:applications} discusses applications of a representation-theoretic flavor.

%%%%%%%%%%%%%%%%%%%%%%%%%%%%%%%%%%%%%%
\section{Central zonotopal algebras}
\label{sec:central_zonotopal}
%%%%%%%%%%%%%%%%%%%%%%%%%%%%%%%%%%%%%%

Let $X\subset \bR^n$ be a finite set of nonzero vectors such that $\mathrm{span}(X)=\bR^n$.
With $v=(v_1,\dots,v_n)\in \bR^n$ we attach the linear form $p_v\coloneqq v_1x_1+\cdots+v_nx_n$, and subsequently define for $Y\subseteq X$ the polynomial
\begin{align}
p_Y=\prod_{y\in Y}p_y.
\end{align}
A \emph{cocircuit} in $X$ is a minimal-under-inclusion subset $Y$  such that  $X\setminus Y$ does not contain a basis for $\bR^n$.
The \emph{cocircuit ideal} $I(X)$ is defined as
\begin{align}
\label{eq:def_cocircuit_ideal}
I(X)\coloneqq \langle p_Y\suchthat Y\subseteq X \text{ a cocircuit}\rangle.
\end{align}
Observe that we could also have defined $I(X)$ to be generated by \emph{all} $p_Y$ where $Y\subseteq X$ are such that $X\setminus Y$ does not contain a basis for $\bR^n$. All such elements are multiples of $p_Y$ for $Y$ a cocircuit.

%Note that it suffices to consider only the minimal cocircuits to generate $I(X)$.

Define the \emph{central $P$-space} associated to $X$ as
\begin{align}
\mc{P}(X)\coloneqq \bQ\left\lbrace p_Y\suchthat \mathrm{rank}(X\setminus Y)=n\right\rbrace.
\end{align}
In other words, $\mc{P}(X)$ is the span of all $p_Y$ over subsets $Y\subset X$ such that $X\setminus Y$ contains a basis.
Let $b(X)$ denote the number of bases in $X$.
We then have the following important results:
\begin{enumerate}
\item $\bQ[\bfx{n}]=\mc{P}(X)\oplus I(X)$ \cite{DR90},
\item $\dim \mc{P}(X)=b(X)$ \cite[Corollary 2.17]{dBDR91}.
\end{enumerate}
See also \cite[Theorem 3.8]{HR11} for these statements in one place.

Recall that with a graded vector space $A\coloneqq \bigoplus_{d\geq 0}A_d$ one can attach its \emph{Hilbert series} $\hilb(A)$ as follows:
\begin{align}
\hilb(A)=\sum_{d\geq 0}\dim(A_d)q^d.
\end{align}
Since $\mc{P}(X)$ is graded by degree, one may ask for its Hilbert series (or equivalently for that of the quotient $\bQ[\bfx{n}]/I(X)$). This is described in terms of the Tutte polynomial $T_{X}(x,y)$ of the linear matroid determined by $X$.
One definition of $T_{X}(x,y)$ involves the notion of internal and external activity.
We will only need the latter for the Hilbert series we are interested in.

Endow the vectors in $X$ with a total order.
For any basis $B$ in $X$, we say that $v\in X\setminus B$ is \emph{externally active} if $v$ is the smallest element in the minimal dependent subset of $B\cup \{v\}$.
Let the total number of externally active elements with respect to $B$ be $\mathrm{ea}(B)$.
Then we have the following result (see \cite[Proposition 4.14]{AP10} or \cite[Theorem 11.8]{DCP}, both of which apply to the quotient $\bQ[\bfx{n}]/\mc{I}(X)$):
\begin{align}
\label{eq:hilb_p(x)}
\hilb(\mc{P}(X))&=\sum_{\text{ bases } B\subset X}q^{|X|-\mathrm{rank}(X)-\mathrm{ea}(B)}\nonumber\\&=
q^{|X|-\mathrm{rank}(X)}T_X(1,q^{-1}).
\end{align}

As mentioned in the introduction, $\mc{P}(X)$ may also be described as a Macaulay-inverse to a \emph{power ideal} determined by $X$. We proceed to describe this next.
Let $H$ be any \emph{hyperplane} in $X$. Recall that a hyperplane is a maximal collection of vectors in $X$ whose span is $(n-1)$-dimensional. This necessarily determines a linear form $\phi_H$ that vanishes on $H$. Let $m_H\coloneqq m_H(X)$ equal the number of vectors in $X$ that do not lie on $H$.
Consider the ideal $\mc{I}(X)$ (see \cite[Section 3.1]{HR11} or \cite[Section 4.1]{AP10}) defined as
\begin{align}
\label{eq:def_power_ideal}
\mc{I}(X)\coloneqq \langle \phi_H^{m_H} \suchthat H\text{ a hyperplane in } X\rangle.
\end{align}
Let $\partial_i=\frac{\partial}{\partial x_i}$.
We then have (\cite[Theorem 3.8 (5)]{HR11}, \cite[Theorem 2.7]{dBDR91}, or \cite[Theorem 11.25]{DCP}):
\begin{align}
\mc{P}(X) = \{f\in \bQ[\bfx{n}]\suchthat p(\partial_1,\dots,\partial_n)f=0 \,\forall p\in \mc{I}(X)\},
\end{align}
which says that $\mc{P}(X)$ is the Macaulay-inverse to $\mc{I}(X)$.

\begin{example}
\label{ex:central_p_space}
\emph{
Suppose $X$ is the collection of vectors in $\bR^2$ determined by the columns of
\[
X=\left[\begin{array}{ccc}1 & 0 & 1\\ 0 & 1 & -1\end{array}\right].
\]
Let us refer to the vectors read from left to right as $v_1$, $v_2$, and $v_3$, and assume the total order $v_1<v_2<v_3$.
The cocircuits are $\{v_1,v_2\}$, $\{v_1,v_3\}$, and $\{v_2,v_3\}$. Thus
\[
 I(X)=\langle x_1x_2, x_1(x_1-x_2),x_2(x_1-x_2) \rangle.
 \]
One checks that
\[
\mc{P}(X)=\bQ\{1,x_1,x_2\}.
\]
 Observe that one can extract exactly 3 bases for $\bR^2$ from $X$ and that $\dim(\mc{P}(X))=3$. Additionally, we have
 \[
 \hilb(\mc{P}(X))=1+2q=q^{3-2}(2+q^{-1}),
 \]
and so the right-hand side agrees with $qT_X(1;q^{-1})$. Indeed, there are two bases $B$ with $\mathrm{ea}(B)=0$ (i.e. \emph{unbroken bases}) and one with $\mathrm{ea}(B)=1$.
}

\emph{Note further that every vector is also a hyperplane in $X$ and the line spanned by any vector does not contain the remaining two vectors. Thus
\[
\mc{I}(X)=\langle x_1^2, x_2^2, (x_1+x_2)^2\rangle.
\]
We leave it to the reader to check that $\mc{P}(X)$ above is indeed the Macaulay-inverse to $\mc{I}(X)$.}
\end{example}

%%%%%%%%%%%%%%%%%%%%%%%
\subsection{The graphical matroid}
\label{subsec:graphic_matroid}
%%%%%%%%%%%%%%%%%%%%%%%

Let $G$ be a connected graph where we allow multiple edges between distinct vertices but forbid self-loops. This assumption will be implicit throughout the article.
Denote the set of edges (respectively vertices) of $G$ by $\edges(G)$ (respectively $\vertices(G)$). We begin by recalling some standard graph-theoretic jargon; see \cite{BM76} for one source.

Given a nonempty proper subset $S\subset \vertices \coloneqq \vertices(G)$, we denote by $G[S]$ the graph induced on vertices in $S$.
We define the \emph{edge cut} $\partial(S)$ of $G$ associated to $S$ to be the set of edges with one endpoint in $S$ and the other in $\bar{S}\coloneqq \vertices \setminus S$.
We denote the cardinality of $\partial(S)$ by $d(S)$.
We denote the set of  cuts by $\cuts(G)$.
A minimal non-empty edge cut is called a \emph{bond}.
We denote the set of bonds by $\bonds(G)$.
The following property \cite[Exercise 2.2.8]{BM76} of bonds in connected graphs is relevant for us: given a nonempty proper $S\subset \vertices$, the cut $\partial(S)$ is a bond if and only if $G[S]$ and $G[\bar{S}]$ are both connected.

We now isolate the case of our interest.
Consider the set of vectors $X_G$ determined by a connected graph $G$ on the vertex set $[n]$.
Let $e_i$ denote the $i$th standard basis vector in $\bR^n$.
For any edge $\{i,j\}\in \edges(G)$ where $i<j$, the vector $e_i-e_j$ belongs to $X_G$ (as many times as there is an edge between $i$ and $j$).
Include the vector $e_1+\cdots+e_n$ as well, thereby
ensuring $\mathrm{span}(X_G)=\bR^n$.

Suppose $Y\subset X_G$ is such that the corresponding set of edges in $G$ has the property that its complement in $\edges(G)$ results in a connected subgraph of $G$.
We say that $Y$ (or the subgraph in $G$ determined by the associated set of edges) is \emph{slim}.
It is clear that if $Y$ is slim, then $X_G\setminus Y$ contains a basis for $\bR^n$.
Indeed one only needs to take a spanning tree in the complement (which is connected) and throw in the vector $e_1+\cdots+e_n$ to get a basis. In particular, slim subgraphs do not give cocircuits.
Recall that a cocircuit is a minimal subset $Y\subset X_G$ such that $X_G\setminus Y$ does not contain a basis.
Clearly, $e_1+\cdots+e_n$ is a cocircuit.
It is straightforward to see that all other cocircuits correspond to collections of vectors corresponding to edges of bonds in $G$.\footnote{Put differently, a bond is code for \emph{minimally non-slim}.}

In summary we have
\begin{align}
I(G)\coloneqq I(X_G)=\langle x_1+\cdots+x_n, p_Y \text{ for } Y \text{ a bond}\rangle.
\end{align}
We thus have
\begin{align}
\mc{P}(G)\coloneqq \mc{P}(X_G)= \bQ\{p_Y\suchthat Y \text{ slim}\}.
\end{align}

We will refer to $\mc{P}(G)$ as the \emph{\textup{(}Postnikov--Shapiro\textup{)} slim-subgraph space} of $G$.
Standard matroid-theoretic techniques allow us to extract a basis for this space after endowing the set of edges with a total order and consider externally active edges for spanning trees.

We now describe the power ideal $\mc{I}(G)\coloneqq \mc{I}(X_G)$ as well. To this end we need to describe all hyperplanes in $X_G$, i.e. all maximal subsets of $X_G$ such that the resulting rank is $n-1$.
These are exactly the complements of cocircuits.

Henceforth set $x_S\coloneqq \sum_{i\in S}x_i$.
Consider first the cocircuit $Y$ determined by $e_1+\cdots+e_n$. The vectors in the complement $X_G\setminus Y$ all lie on the hyperplane $x_{[n]}=0$. Thus we get that $x_{[n]}\in \mc{I}(G)$.
Now consider the cocircuit determined by vectors corresponding to edges of a cut $\partial(S)$.
It is immediate that collection of vectors in $X_G\setminus Y$ all lie on the hyperplane
\begin{align*}
|\bar{S}|\sum_{i\in S}x_i-|S|\sum_{i\in \bar{S}}x_i=0.
\end{align*}
Taking into account that $x_{[n]}\in\mc{I}(G)$, we thus  infer that
$x_{S}^{d(S)}\in \mc{I}(G)$,
and, by symmetry, that
$x_{\bar{S}}^{d(S)}\in \mc{I}(G)$.
We have thus obtained the following presentation:
\begin{align}
\mc{I}(G)=\langle x_{[n]},  x_S^{d(S)} \text{ where $\partial(S)\in\bonds(G)$}\rangle.
\end{align}

\begin{example}
\emph{Consider $G$ in Figure~\ref{fig:k4-e} with $\vertices(G)=[4]$. Then $X_G$ is:
\begin{align*}
\left[\begin{array}{cccccc}1 & 1 &0 &0 &0 & 1\\
-1 & 0 &1 &1 &0 & 1\\
0 & 0 &-1 &0 &1 & 1\\
0& -1 &0 &-1 &-1 & 1\\
\end{array}\right].
\end{align*}
In this case, we have
\[
\mc{I}(G)=\langle x_1^2,x_2^3,x_3^2,x_4^3, (x_1+x_2)^3,(x_1+x_4)^3,x_1+x_2+x_3+x_4\rangle.
\]
It may be checked that
\[
\hilb\left(\bQ[\bfx{4}]/\mc{I}(G)\right)=1+3q+4q^2,
\]
which evaluates to $8$ at $q=1$. The latter is also the number of spanning trees of $G$.
}
\end{example}
\begin{figure}
\includegraphics[scale=0.7]{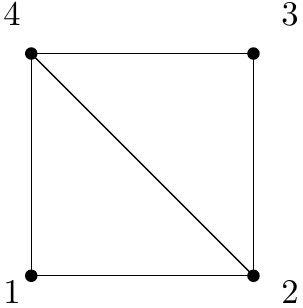}
\caption{The graph $K_4-\text{edge}$.}
\label{fig:k4-e}
\end{figure}

The quotient $\bQ[\bfx{n}]/\mc{I}(G)$ is essentially what Postnikov--Shapiro \cite[\S 3]{PS03} call $\mc{B}_G$ in their seminal article, even though the presentation here appears different.
Indeed, the Postnikov--Shapiro ideal only involves the variables $x_1$ through $x_{n-1}$, but this may be achieved by eliminating the variable $x_n$ given the linear relation $x_1+\cdots+x_n\in \mc{I}(G)$.
Other than that, observe that Postnikov--Shapiro consider linear forms coming from \emph{all} subsets $S\subseteq[n-1]$ whereas we have only included those $S$ that give rise to bonds $\partial(S)$.
It turns out that the resulting ideals are in fact the same; this follows from applying \cite[Theorem 4.17]{AP10} with $\mc{A}$ equal to the collection of vectors coming from edges of $G$ and $k$ equal to $-1$.\footnote{This is the value of $k$ that corresponds to the central zonotopal algebra. The cases $k=0$ and $k=-2$ give the \emph{external} and \emph{internal} zonotopal algebras respectively.}

Postnikov and Shapiro give a monomial basis for $\mc{B}_G$ indexed by \emph{$G$-parking functions} \cite[Theorem 3.1]{PS03}, and show that the Hilbert series is a specialization of the Tutte polynomial of $G$ \cite[Theorem 3.3]{PS03}.
It is also known that \emph{$G$-parking functions} are \emph{$q$-reduced divisors} on $G$ (once a distinguished vertex $q$ is chosen); it is a different set of divisors on graphs that is relevant for our purposes and we introduce those in Section~\ref{sec:orbit_harmonics}.

%%%%%%%%%%%%%%%%%%%%%%
\subsection{Two variations on this theme}
\label{subsec:variations}
%%%%%%%%%%%%%%%%%%%%%%
Our treatment of the graphical matroid is slightly non-standard, and we discuss two other ways in which it is usually presented. As we hope to demonstrate, all perspectives have their benefits.

In the first perspective, for a connected graph $G$ with $\vertices(G)$ identified with $[n]$, one constructs an $n\times |\edges(G)|$ matrix with $e_i-e_j$ for edge $\{i<j\}\in \edges(G)$. This is a rank $n-1$ matrix with columns in $\bR^n$. If one strikes out the last row (i.e. the row corresponding to the vertex $n$ in $G$), then one obtains a full rank matrix with columns in $\bR^{n-1}$. If we denote this collection of vectors in $\bR^{n-1}$ by $\check{X}_G$, then it is not hard to verify that the power ideal $\check{\mc{I}}(G)\coloneqq \mc{I}(\check{X}_G) \subset \bQ[x_1,\dots,x_{n-1}]$ has the following presentation:
\begin{align}
\check{\mc{I}}(G)=\langle  x_S^{d(S)} \text{ where $\partial(S)$ is a bond with $S\subseteq [n-1]$}\rangle.
\end{align}
From this we have the following isomorphism
\begin{align}
\label{eq:strike_out_last_row}
\bQ[\bfx{n}]/\mc{I}(G) \cong \bQ[\bfx{n-1}]/\check{\mc{I}}(G),
\end{align}
with the latter being more in the spirit of \cite{PS03}. The reader is welcome to check that Example~\ref{ex:central_p_space} applies this construction with $G$ equal to $K_3$. We record an observation here (to be addressed in more detail later): for $G=K_n$, the quotient on the right hand side carries an $S_{n-1}$-action, but `hides' the $S_n$-action, unlike the quotient on the left hand side.
\begin{remark}
  \emph{There is nothing special about striking out the last row. Indeed one can strike out \emph{any} row and obtain a rank $n-1$ matrix that determines an isomorphic matroid, and a quotient isomorphic to that in~\eqref{eq:strike_out_last_row}.}
\end{remark}

We now describe our second perspective. The reader may find our insistence on working with a collection of vectors $X$ in $\bR^n$ that span $\bR^n$ as stringent.
Indeed as discussed in the beautiful book of De Concini--Procesi \cite{DCP}, one can work with a collection $X$ that spans an $s$-dimensional real space $V$. All results in this section still hold with the simple adjustment that one work with polynomial functions on $V$ (in the case $V=\bR^n$, this translates to working in $\bQ[\bfx{n}]$.)

With this guidance, we now return to our collection $\hat{X}_G\coloneqq X_G\setminus \{e_1+\cdots+e_n\}$. In other words we omit the vector that was added to ensure $\mathrm{span}(X_G)=\bR^n$. Since $G$ is connected, we have that $\mathrm{span}(\hat{X}_G)$ is the hyperplane $V$ in $\bR^n$ cut out by $x_1+\cdots+x_n=0$.
The ring of regular functions on $V$ can be identified by the polynomial ring $\bQ[x_i-x_j\suchthat i<j]$.
We now have
\begin{align}
\hat{I}(G)\coloneqq I(\hat{X}_G)&=\langle p_Y \text{ for } Y \in \bonds(G)\rangle\\
\mc{P}(G)=\mc{P}(\hat{X}_G)&=\bQ\{p_Y\suchthat Y \text{ slim}\}.
\end{align}
Note here that the ideal $\hat{I}(G)$ is taken in $\bQ[x_i-x_j\suchthat i<j]$.
By the general result \cite[Corollary 11.23]{DCP}, we have the direct sum decomposition
\begin{align}
\label{eq:central_p_decomposition}
\bQ[x_i-x_j\suchthat i<j]=\mc{P}(G)\oplus \hat{I}(G).
\end{align}
Note that the $\bQ$-vector spaces $\mc{P}(X_G)$ and $\mc{P}(\hat{X}_G)$ are in fact the same, which is why we have chosen to refer to both as $\mc{P}(G)$.

In summary, throwing in the vector $e_1+\cdots+e_n$ is harmless to our story, and omitting it will make the link to Efimov's construction in the context of the COHA of symmetric quivers a bit more transparent.

%%%%%%%%%%%%%%%%%%%%%%%
\section{Orbit harmonics and break divisors}
\label{sec:orbit_harmonics}
%%%%%%%%%%%%%%%%%%%%%%%
\subsection{Break divisors}
 Let $G$ be a connected graph.
 We define the \emph{genus} of $G$ to be $g(G)\coloneqq |\edges(G)|-|\vertices(G)|+1$.
 A \emph{divisor} $D$ on $G$ is an assignment $D:\vertices(G)\to \bZ$.
 The sum $\sum_{v\in \vertices(G)}D(v)$ is called the \emph{degree} of $D$ and denoted by $\deg(D)$.
 If $D(v)\geq 0$ for all $v\in \vertices(G)$, then we say that $D$ is \emph{effective}.
If $H$ is a connected subgraph of $G$, then a divisor $D$ on $G$ naturally restricts to a divisor $D|_H$ on $H$.
An effective divisor $D$ of degree $g(G)$ is said to be a \emph{break} divisor if for every connected subgraph $H$ of $G$ we have
\begin{align}
\deg(D|_H)\geq g(H).
\label{eq:def_break_divisor}
\end{align}
We denote by $\brkd(G)$ the set of break divisors on $G$.
It is known that $|\brkd(G)|$ equals the number of spanning trees on $G$ \cite[Theorem 4.25]{ABKS14}.

By identifying $\vertices(G)$ with $[n]$, we can, and will, interpret elements of $\brkd(G)$ as lattice points in $\bR^{n}$.
Consider the \emph{graphical zonotope} $\mc{Z}_G\in\bR^n$ determined by $G$ by consider the Minkowski sum of line segments $[e_i,e_j]$ for every edge $\{i,j\}\in E(G)$.
Let $\Delta_{n-1,n}$ denotes the $(n-1)$-th hypersimplex determined by taking the convex hull of the $S_n$-orbit of the point $(1^{n-1},0)\in \bR^n$.
This given, break divisors are lattice points in the \emph{trimmed graphical zonotope} $\mc{Z}_G-\Delta_{n-1,n}$ \cite[Proposition 2.1]{KRT21}. %Furthermore, this latter lattice polytope is a generalized permutahedron.
\begin{example}
\label{ex:break_divisors_orbits}
\emph{
Consider the graph $G$ in Figure~\ref{fig:k23}.
Elements of $\brkd(G)$ are tuples $(a,b,c,d,e)\in \mathbb{Z}_{\geq 0}^5$ that satisfy \eqref{eq:def_break_divisor}.
Any subgraph that is a tree gives a trivial constraint. Furthermore, note the $S_2\times S_3$ symmetry in $G$. So given any break divisor we can permute $a$ and $b$ (respectively $c$, $d$, and $e$) to get another break divisor. Thus (up to $S_2\times S_3$ symmetry) the only relevant inequality is
\begin{align*}
a+b+c+d\geq 1
\end{align*}
on top of the equality $a+b+c+d+e=2$. In summary we obtain four orbit representatives for $S_2\times S_3$-action on $\brkd(G)$:
\begin{align*}
(2,0,0,0,0), (1,1,0,0,0), (1,0,1,0,0), (0,0,1,1,0).
\end{align*}
We thus obtain $|\brkd(G)|=12$, which agrees with the number of spanning trees of $G$.\footnote{Recall that the number of spanning trees of the complete bipartite graph $K_{m,n}$ equals $m^{n-1}n^{m-1}$.}}
\end{example}

\begin{figure}
\includegraphics[scale=0.8]{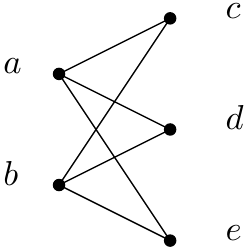}
\label{fig:k23}
\caption{The complete bipartite graph $K_{2,3}$.}
\end{figure}

In the example earlier, we incorporated the $S_2\times S_3$ symmetry into our computation to make things concise. In general, from the inequalities defining break divisors, it is clear that  $\aut{G}$ acts on $\brkd(G)$ by permutations.\footnote{It is entirely possible that a larger group acts on $\brkd(G)$. For instance if one takes an $n$-cycle, then $\aut{G}$ is the cyclic group $\bZ/n\bZ$. On the other hand, the set of break divisors is given by the vertices of the standard simplex, and thus enjoys $S_n$-symmetry.}
The situation is ideal for invoking the method of orbit harmonics.

\subsection{Orbit harmonics}
The method of {\em orbit harmonics} gives a technique for turning an ungraded permutation representation of some linear group $\mathscr{G}$
acting on a finite point
locus $Y$ into a graded $\mathscr{G}$-module.
Orbit harmonics was introduced by Kostant \cite{Kostant} and has seen subsequent application (for example)
by Garsia--Procesi \cite{GP92} in the context of Springer fibers and by
Haglund--Rhoades--Shimozono \cite{HRS18} in the context of Macdonald-theoretic delta
operators. See \cite{Rh22} for a survey of recent developments.

%We begin with a concise description of the point-orbit method.
Consider a finite point set $Y \subset \bQ^n$ and let $\mathsf{I}(Y) \subseteq \bQ[x_1, \dots, x_n]$
be the ideal of polynomials which vanish on $Y$. The quotient of $\bQ[x_1, \dots, x_n]$ by the ideal
$\mathsf{I}(Y)$ has vector space dimension $\dim_{\bQ}\left(\bQ[\bfx{n}]/\mathsf{I}(Y)\right)=|Y|$.
If the set $Y$ is $\mathscr{G}$-stable for some group $\mathscr{G}\subset \mathrm{GL}_n(\bQ)$, then we have an isomorphism of ungraded
$\mathscr{G}$-modules
\[
\bQ[\bfx{n}]/\mathsf{I}(Y)\cong_{\mathscr{G}} \bQ[Y].
\]
 Orbit harmonics produces a quotient that will afford the structure of a {\em graded} $\mathscr{G}$-module.
 Given a nonzero polynomial $f \in \bQ[x_1, \dots, x_n]$, let $\tau(f)$ be the top degree homogeneous component of $f$.
 That is, if $f = f_d + \cdots + f_1 + f_0$ with $f_i$ homogeneous of degree $i$ and $f_d \neq 0$, then we have $\tau(f) = f_d$.
 Define an ideal $\mathsf{T}(Y) \subseteq \bQ[\bfx{n}]$ by
 \[
 \mathsf{T}(Y) = \langle \tau(f) \,:\, f \in \mathsf{I}(Y), \, \, f \neq 0 \rangle.
 \]
The $\mathscr{G}$-module isomorphism given above extends to a chain
\[
\bQ[\bfx{n}]/\mathsf{T}(Y) \cong_{\mathscr{G}}
\bQ[\bfx{n}]/\mathsf{I}(Y)\cong_{\mathscr{G}} \bQ[Y],
\]
with the added feature that the left hand side is a graded $\mathscr{G}$-module.

%%%%%%%%%%%%%%%%%%%%%%%%%%
\subsection{Applying orbit harmonics to $\brkd(G)$}
\label{subsec:gp_to_brkd}
%%%%%%%%%%%%%%%%%%%%%%%%%%
Throughout this subsection, our group $\mathscr{G}$ will be the automorphism group $\aut{G}$ of our graph $G$ and our point set $Y$
carrying the $\aut{G}$ action will be the set $\brkd(G)$ of break divisors of $G$.
For brevity we will refer to the ideals $\mathsf{I}(\brkd(G))$ and $\mathsf{T}(\brkd(G))$ by $\mathsf{I}(G)$ and $\mathsf{T}(G)$ respectively.
Our strategy is as follows.
\begin{enumerate}
\item Produce a family of polynomials indexed by $\bonds(G)$ that vanish on $\brkd(G)$.
\item Consider $\tilde{\mathsf{T}}(G)$ the ideal generated by top degree homogeneous summands from the polynomials in the aforementioned family.
\item Use the inclusion $\tilde{\mathsf{T}}(G) \subset \mathsf{T}(G)$ to infer on the one hand that
\[
\dim\left(\bQ[\bfx{n}]/\tilde{\mathsf{T}}(G)\right) \geq \dim\left(\bQ[\bfx{n}]/\mathsf{T}(G)\right).
\]
On the other hand, it will transpire that the $\tilde{\mathsf{T}}(G)$ equals the power ideal $\mc{I}(G)$ and thus its dimension equals the number of spanning trees in $G$. Additionally, the right hand side equals $|\brkd(G)|$. This will allow us to conclude that $\tilde{\mathsf{T}}(G)=\mathsf{T}(G)$.
\end{enumerate}

\begin{proposition}
\label{prop:GP_break_divisors}
We have the equality of ideals
\[
\mathsf{T}(G)=\mc{I}(G).
\]
\end{proposition}
\begin{proof}
As before, set $x_S\coloneqq \sum_{i\in S}x_i$ for $S\subset \vertices(G)=[n]$.
Since the degree of any break divisor is $g(G)$, we know that $x_{[n]}-g(G)\in \mathsf{I}(G)$.
Now consider a nonempty proper $S\subset \vertices(G)$ such that $\partial(S)\in \bonds(G)$.  We will be interested in the values $x_S(p)$ as $p$ ranges over $\brkd(G)$.

Since $G[S]$ is connected, we  get the lower bound
\begin{align}
x_S(p)\geq g(G[S]).
\end{align}
Since $G[\bar{S}]$ is also connected, and because $x_S(p)+x_{\bar{S}}(p)=g(G)$, we get the upper bound
\begin{align}
x_S(p)\leq g(G)-g(G[\bar{S}]).
\end{align}
Thus we see that $x_S(p)\in [g(G[S]), g(G)-g(G[\bar{S}])]$.

Observe also that there are $g(G)-g(G[S])-g(G[\bar{S}])+1$ integers in this interval.
This quantity is exactly the size $d(S)$ of the bond $\partial(S)$. Indeed, using the definition of genus, one may check that
\begin{align}
g(G)-g(G[S])-g(G[\bar{S}])+1&=|\edges(G)|-|\vertices(G)|-|\edges(G[\bar{S}])|+|\vertices(G[\bar{S}])|-|\edges(G[S])|+|\vertices(G[S])|\nonumber\\
&=|\edges(G)|-|\edges(G[\bar{S}])|-|\edges(G[S])|=d(S).
\end{align}

We now proceed to show that $x_S(p)$ achieves all values in $[g(G[S]),g(G[S])+d(S)]$.
\emph{While we do not need the full strength of this preceding statement to establish the proposition, we have retained it in the hope that it might be of independent combinatorial interest.}

We  first show that $x_S$ achieves the values $g(G)$ and $g(G)+d(S)$.
To this end, we need the following alternative description \cite[Section 3]{ABKS14} (see also \cite[Section 2.2]{BW18}) for break divisors on $G$ given a spanning tree $T$.
For every edge in $\edges(G)\setminus \edges(T)$, assign a chip to one of its two endpoints. The sequence tracking the number of chips at each vertex is always  a break divisor.
Note also that flipping a given assignment across an edge produces another break divisor, which differs from the initial break divisor by an appropriate type $A$ root $e_i-e_j$.
We thus get $2^{|\edges(G)\setminus \edges(T)|}$ break divisors for any choice of spanning tree $T$.

For the bond $\partial(S)$, pick spanning trees $T_{S}$ and $T_{\bar{S}}$ for $G[S]$ and $G[\bar{S}]$ respectively.
Let $e\in \partial(S)$ be any edge such that $T_S\cup \{e\}\cup T_{\bar{S}}$ is a spanning tree in $G$.
We now apply the procedure above to produce a break divisor $p$ with $x_S(p)=g(G[S])$.

Note that
\begin{align}
\edges(G)\setminus \edges(T)=(\edges(G[S])\setminus \edges(T_S)) \cup (\partial(S)\setminus \{e\})\cup (\edges(G[\bar{S}])\setminus \edges(T_{\bar{S}})).
\end{align}
Consider edges  in $\edges(G)\setminus \edges(T)$ depending on which of the three sets on the right-hand side they belong to.
If an edge is in $(\partial(S)\setminus \{e\})$, then assign a chip to the endpoint in $\bar{S}$. Otherwise, assign randomly.
Since vertices in $S$ get assigned chips only from edges in $\edges(G[S])\setminus \edges(T_S)$, our procedure is guaranteed to have a produced a break divisor $p$ on $G$ that restricts to a break divisor on $G[S]$.  Thus we must have $x_S(p)=g(G[S])$.

A symmetric argument with the change that for edges  in $(\partial(S)\setminus \{e\})$, assign a chip to the endpoint in $S$ produces a break divisor $p$ on $G$ that restricts to a break divisor on $G[\bar{S}]$. This in turn means $x_S(p)=g(G)-g(G[\bar{S}])$. Thus at this stage we have produced break divisors that attain the bounds $g(G[S])$ and $g(G)-g(G[\bar{S}])$ respectively.

It is now easy to see how one may interpolate between these extremes one step at a time whilst incrementing $x_S$ by $1$ at each step.
To attain the value $g(G[S])+k$ where $0< k< d(S)-1$, for $k$ edges in $\partial(S)\setminus \{e\}$, assign a chip to the endpoint in $S$. Otherwise assign to the endpoint in $\bar{S}$.

It follows from the preceding discussion that
\[
\prod_{i=g(G[S])}^{g(G)-g(G[\bar{S}])} (x_S-i)\in \mathsf{I}(G).
\]
for any nonempty proper $S\subset \vertices(G)$ for which $\partial(S)\in \bonds(G)$.
Next note that the top degree homogeneous summands from such polynomials above generate exactly the power ideal $\mc{I}(G)$. Thus we must have  the inclusion
\begin{align}
\mc{I}(G)\subset \mathsf{T}(G),
\end{align}
which in turn implies the inequality
\begin{multline}
\text{ \# of spanning trees}= \dim\left( \bQ[\bfx{n}]/\mc{I}(G) \right)\geq \dim\left( \bQ[\bfx{n}]/\mathsf{T}(G) \right)=|\brkd(G)|.
\end{multline}
Given that the quantities on either extreme are in fact equal \cite[Theorem 4.25]{ABKS14}, the claim follows.
\end{proof}

\begin{theorem}
\label{thm:piecing_things}
We have the following isomorphisms and equalities of ungraded $\aut{G}$-modules:
\[
\bQ[\brkd(G)]\cong \bQ[\bfx{n}]/\mathsf{T}(G) = \bQ[\bfx{n}]/\mc{I}(G) \cong \mc{P}(G).
\]
where the middle equality and right isomorphism are in the category of \emph{graded} $\aut{G}$-modules.
\end{theorem}
\begin{proof}
The first isomorphism comes from orbit harmonics. The second equality is exactly Proposition~\ref{prop:GP_break_divisors}. The last isomorphism follows since $\mc{P}(G)$ is the Macaulay-inverse system of $\mc{I}(G)$.
\end{proof}
\begin{remark}
\emph{
 Theorem~\ref{thm:piecing_things} provides an algebraic perspective on the fact that the set of $q$-reduced divisors on $G$ (essentially $G$-parking functions) has the same cardinality as $\brkd(G)$. Indeed, as was mentioned earlier, by work of Postnikov--Shapiro, the quotient $\bQ[\bfx{n}]/\mc{I}(G)$ has a monomial basis indexed by $G$-parking functions.  On the other hand we just showed that this same quotient arises when applying the point orbit method to the set of break divisors on $G$.
}
\end{remark}

%%%%%%%%%%%%%%%%%%%%%%%%%%%%%%%%%%%%%
\subsection{External and internal zonotopal algebras via orbit harmonics}
\label{subsec:more_orbit_harmonics}
%%%%%%%%%%%%%%%%%%%%%%%%%%%%%%%%%%%%%
 It is natural to ask whether two other well-known zonotopal algebras \textemdash{} \emph{external} and \emph{internal} \textemdash{} arise in this story.
 We demonstrate how these algebras arise by applying orbit harmonics to graphical zonotopes. \emph{It bears repeating that the central case is central to this article, and we will have no further use for the results of this subsection.}

 We begin by establishing some notation reminiscent of that employed in the central case.
 As usual, fix $G=([n],E)$.
 Consider the  ideals:
 \begin{align}
 \label{eq:external_internal}
 \mc{I}_+(G)&=\langle x_{[n]},  x_S^{d(S)+1} \text{ where $\partial(S)\in\bonds(G)$}\rangle\\
 \mc{I}_-(G)&=\langle x_{[n]},  x_S^{d(S)-1} \text{ where $\partial(S)\in\bonds(G)$}\rangle.
 \end{align}
The external and internal zonotopal algebras are defined as respective quotients of $\bQ[\bfx{n}]$ by these ideals.
The dimensions of these algebras are conveniently described in terms of lattice points in the graphical zonotope $\mc{Z}_G$.
 Let $\mc{Z}_G^{\circ}$ denote the interior of $\mc{Z}_G$.
It is known \cite[Proposition 1.1]{HR11} that
\begin{align}
  \label{eq:external_dimension}
  \dim(\bQ[\bfx{n}]/\mc{I}_+(G))&=|\mc{Z}_G\cap \bZ^n|\\
  \label{eq:internal_dimension}
  \dim(\bQ[\bfx{n}]/\mc{I}_+(G))&=|\mc{Z}_G^{\circ}\cap \bZ^n|.
\end{align}

We will now realize these quotients by  applying orbit harmonics.
Like before, our point sets on which $\aut{G}$ acts are given to us. Let us recast them in the language of divisors to fit with the theme in this article.

Fix an orientation $\mc{O}$ of the edges of $G$.
This determines the \emph{orientable divisor} \[
(\mathrm{indeg}_{\mc{O}}(1)-1,\dots,\mathrm{indeg}_{\mc{O}}(n)-1)\] where $\mathrm{indeg}_{\mc{O}}(i)-1$ is one less than the number of directed edges pointing to $i$.
Note that the degree of this divisor is $|E|-|V|=g(G)-1$.
Denote the set of orientable divisors on $G$ by $\orid(G)$.
By interpreting orientable divisors as lattice points in $\bZ^{n}$, we obtain a polytope in $\bR^{n}$ which is the translation of $\mc{Z}_G$ by the vector $(1,\dots,1)\in \bR^n$. Note that it is  contained in the hyperplane $x_1+\cdots+x_n=g(G)-1$.
We refer to the set of \emph{interior} lattice points in this polytope by $\orid^{\circ}(G)$.
Clearly, both $\orid(G)$ and $\orid^{\circ}(G)$ carry $\aut{G}$-actions.

Now define the ideal $\mathsf{T}_+(G)$ (respectively $\mathsf{T}_-(G)$) as resulting from applying orbit harmonics to $\orid(G)$ (respectively $\orid^{\circ}(G)$), in analogy to $\mathsf{T}(G)$ resulting from $\brkd(G)$.

\begin{proposition}
  We have the equality of ideals
\begin{align*}
  \mathsf{T}_+(G)&=\mc{I}_+(G),\nonumber\\
  \mathsf{T}_-(G)&=\mc{I}_-(G).
\end{align*}
\end{proposition}
\begin{proof}
  We consider the first equality. The proof is similar to that for Proposition~\ref{prop:GP_break_divisors}.
  Since the degree of any orientable divisor is $g(G)-1$, we know that $x_1+\cdots+x_n\in \mathsf{T}_+(G)$. Now fix a proper nonempty $S\subset \vertices(G)$. Consider the linear form $x_S$ as we range over $\orid(G)$. By picking any orientation with the property that all edges in the cut $\partial(S)$ go from $\vertices(G[S])$ to $\vertices(G[\bar{S}])$, we are guaranteed that $x_S$ attains the value $g(G[S])-1$. This is also clearly the minimum. By reversing the orientation on all edges in the cut, we see that $x_S$ attains the value $g(G[S])-1+d(S)$, and that is the maximum such.
  Thus we can see that $x_S$ takes (all) values in the interval $[g(G[S])-1,g(G[S])-1+d(S)]$, which has size $d(S)+1.$
  It follows that $x_S^{d(S)+1}\in \mathsf{T}_+(G)$.
  The equality $\mathsf{T}_+(G)=\mc{I}_+(G)$ now follows from the fact that $|\orid(G)|=|\mc{Z}_G\cap \bZ^n|$ and equation~\eqref{eq:external_dimension}.

  The internal case is essentially the same. One only needs to note that for an orientable divisor corresponding to an interior point in $\orid^{\circ}(G)$, the minimum and maximum values attained by the linear form $x_S$ are $(g(G[S])-1)+1$ and $(g(G[S])-1)+(d(S)-1)$ respectively. This in turn implies that $x_S^{d(S)-1}\in \mathsf{T}_-(G)$.
\end{proof}

We conclude this subsection with one important instance.
Consider $G=K_n^m$ for $m\geq 1$.
Then $\mc{Z}_G$ is the $m$-fold dilation of the standard permutahedron in $\bR^n$ whose vertices are given by permutations of $(n-1,\dots,1,0)$.
The number of lattice points in $\mc{Z}_{G}$ is given by the number of forests on $G$. Orbit harmonics tells us that the permutation action on these lattice points (or equivalently, $\orid(G)$) can be used to construct a graded $S_n$-module given by the quotient \[
\bQ[\bfx{n}]/\langle (x_{i_1}+\cdots x_{i_k})^{mk(n-k)+1} \text{ for } 1\leq i_1<\cdots<i_k\leq n \text{ where } k\in [n] \rangle.
\]
At $m=1$, this quotient was introduced by Shapiro--Shapiro in a geometric context  \cite{ShSh98}  (see also \cite{PSS99}), and the problem of determining its graded $S_n$-module structure is posed \cite[Problem 3]{ShSh98}.

%%%%%%%%%%%%%%%%%%%%%%%%%%%%%%%%%
\section{Efimov's construction}
\label{sec:Efimov}
%%%%%%%%%%%%%%%%%%%%%%%%%%%%%%%%%

Let $A=(a_{ij})_{i,j\in [k]}$ be a symmetric matrix with nonnegative entries. We assume $a_{ii}\geq 1$ throughout.
The matrix $A$ determines a quiver $Q$ on $k$ vertices labeled $1$ through $k$, with $a_{ij}$ arrows from $i$ to $j$ for $i,j\in [k]$.
We assume that $A$ is such that $Q$ is connected.
Let $\gamma=(\gamma_1,\dots,\gamma_k)\in \bZ_{\geq 0}^k$ be a dimension vector.
We begin by describing the essential construction in \cite{Efi}.

%%%%%%%%%%%%%%%%%%%%%%%%%%%%%%%%%%
\subsection{Quantum Donaldson--Thomas (DT) invariants}
\label{subsec:quantum_dt}
%%%%%%%%%%%%%%%%%%%%%%%%%%%%%%%%%%

We consider variables $x_{i,\alpha}$ for $1\leq i\leq k$ and $1\leq \alpha\leq \gamma_i$, and let $\sigma_{\gamma}$ be their sum.
Define the polynomial ring
\begin{align}
A_{\gamma}=\bQ[x_{i,\alpha}\suchthat 1\leq i\leq k,1\leq \alpha\leq \gamma_i].
\end{align}
Furthemore, let
\begin{align}
A_{\gamma}^{\prim}=\bQ[x_{j,\alpha_2}-x_{i,\alpha_1}\suchthat 1\leq i,j\leq k,1\leq \alpha_1\leq \gamma_i, 1\leq \alpha_2\leq \gamma_j].
\end{align}
Then we have $A_{\gamma}=A_{\gamma}^{\prim}\otimes \bQ[\sigma_{\gamma}]$.
Define $S_{\gamma}\coloneqq S_{\gamma_1}\times \cdots \times S_{\gamma_k}$.

Define $J_{\gamma}$ to be the smallest $S_{\gamma}$-stable $A_{\gamma}^{\prim}$-submodule of $A_{\gamma}^{\prim}$ such that the following holds: for any decomposition of $\gamma=\delta+\bar{\delta}$ where both $\delta$ and $\bar{\delta}$ are nonzero we have that
\begin{align}
  \label{eq:important_poly}
f_{\delta,\bar{\delta}}=\displaystyle\prod_{i\neq j\in[k]}\prod_{\alpha_1=1}^{\delta_i}\prod_{\alpha_2=\delta_j+1}^{\gamma_j}(x_{j,\alpha_2}-x_{i,\alpha_1})^{a_{ij}}\displaystyle\prod_{i\in[k]}\prod_{\alpha_1=1}^{\delta_i}\prod_{\alpha_2=\delta_i+1}^{\gamma_i}(x_{i,\alpha_2}-x_{i,\alpha_1})^{a_{ii}-1} \in J_{\gamma}.
\end{align}
 We shall reinterpret $f_{\delta,\bar{\delta}}$ in terms of cuts (as the indexing hints) in a graph constructed from $(A,\gamma)$.

Letting $\mc{H}_{\gamma}^{\prim}\coloneqq (A_{\gamma}^{\prim})^{S_{\gamma}}$, consider the  decomposition \cite[p. 1139]{Efi}
\begin{align}
\label{eq:Efimov_decomposition}
\mc{H}_{\gamma}^\prim=V_{\gamma}^{\prim}\oplus J_{\gamma}^{S_{\gamma}}.
\end{align}
The quantum DT invariants of the quiver $Q$ with dimension vector $\gamma$ (assuming trivial potential and stability) arise as dimensions of the graded pieces of the graded vector space $V_{\gamma}^{\prim}$ as explained in \cite[Section 4]{Efi}. We describe the $\bZ$-grading employed.

The \emph{Euler form} $\chi_Q(\gamma,\delta)$ given dimension vectors $\gamma$ and $\delta$ is defined as
\begin{align}
\chi_Q(\gamma,\delta)\coloneqq \sum_{1\leq i\leq k}\gamma_i\delta_i-\sum_{1\leq i,j\leq k}a_{ij}\gamma_i\delta_j.
\end{align}
Polynomials $f\in \mc{H}_{\gamma}$ of degree $k$ get assigned the grading $2k+\chi_Q(\gamma,\gamma)$.
We let $V_{\gamma,k}^{\prim}$ be the space of elements in $V_{\gamma}^{\prim}$ with this grading.
Following \cite[Section 4]{Efi}, set
\begin{align}
c_{\gamma,k}\coloneqq \dim(V_{\gamma,k}^{\mathrm{prim}}),
\end{align}
and consider the polynomial in $\bZ_{\geq 0}[q^{\pm \frac{1}{2}}]$:
\begin{align}
\Omega_{\gamma}(q)=\sum_{k\in \bZ}c_{\gamma,k}q^{k/2}.
\end{align}
These $\Omega_{\gamma}(q)$ are the \emph{quantum DT-invariants} of the quiver $Q$.

We get a necessary condition for when $V_{\gamma,k}^{\prim}$ is nonzero in \cite[Theorem 1.2]{Efi} (also presented as \cite[Theorem 3.10]{Efi}). To state it we need the quantity $N_{\gamma}(Q)$ \cite[Section 1]{Efi} defined\footnote{The original definition looks slightly different as we have simplified the expression using our assumption $a_{ii}\geq 1$.} as
\begin{align}
\label{eq:n_gamma}
N_{\gamma}(Q)\coloneqq \frac{1}{2}\left( \sum_{1\leq i\neq \leq k}a_{ij}\gamma_i\gamma_j+\sum_{1\leq i\leq k}(a_{ii}-1)\gamma_i(\gamma_i-1)\right)-\sum_{1\leq i\leq k}\gamma_i +2.
\end{align}
This given, the following holds.
\begin{theorem}[{\cite[Theorem 3.10]{Efi}}]
\label{thm:k_for_nonzero_spaces}
If $V_{\gamma,k}^{\mathrm{prim}}\neq 0$, then $\gamma\neq 0$, and
\[
k\equiv \chi_Q(\gamma,\gamma) \mod 2\quad \text{and} \quad \chi_Q(\gamma,\gamma)\leq k\leq \chi_Q(\gamma,\gamma)+2N_{\gamma}(Q).
\]
\end{theorem}
In view of Theorem~\ref{thm:k_for_nonzero_spaces} we may rewrite $\Omega_{\gamma}(q)$ as
\begin{align}
\Omega_{\gamma}(q)=q^{\frac{1}{2}\chi_Q(\gamma,\gamma)}\sum_{0\leq k\leq N_{\gamma}(Q)-1}c_{\gamma,2k+\chi_Q(\gamma,\gamma)}q^{k}.
\end{align}
Let us denote the sum on the right by $\tilde{\Omega}_{\gamma}(q)$. This lies in $\bZ_{\geq 0}[q]$ and its degree is bounded above by $N_{\gamma}(Q)-1$. We will abuse notation and refer to $\tilde{\Omega}_{\gamma}(q)$ as the quantum DT invariant as well.\footnote{As the careful reader may have realized, we already do so in the introduction.}

The reader might find the preceding  construction both ingenious and mysterious.
It transpires that the space $V_{\gamma,k}^{\prim}$ is the space of $S_\gamma$-invariants for a central $P$-space determined rather naturally from the quiver and the dimension vector.
We proceed to describe this construction.

%%%%%%%%%%%%%%%%%%%%%%%%%%%%%%%%%%%%%%%%%
\subsection{The covering graph construction}
\label{subsec:symmetric_matrix_to_graph}
%%%%%%%%%%%%%%%%%%%%%%%%%%%%%%%%%%%%%%%%
Our point of departure from Efimov is to consider an analogue of  \eqref{eq:Efimov_decomposition} where we do not  take $S_{\gamma}$-invariants.
We will construct an $S_\gamma$-stable  space $W_{\gamma}^{\prim}$ so that the following holds:
\begin{align}
\label{eq:general_efimov}
A_{\gamma}^{\prim}=W_{\gamma}^{\prim}\oplus J_{\gamma}.
\end{align}
Our next construction is crucial to this end.

%Note that $W_{\gamma}^{\prim}$ naturally affords an $S_{\gamma}$-representation. We will realize this space as a central $P$-space.
%Of course, $V_{\gamma}^{\prim}=(W_{\gamma}^{\prim})^{S_{\gamma}}$.
% We proceed to show that $W_{\gamma}^{\prim}$ is the central $P$-space for the vector configuration of the sort considered in Subsection~\ref{subsec:graphic_matroid}.

Given the symmetric quiver $Q$ as before, we construct an undirected graph $G_{Q,\gamma}$ as follows.
Consider a set of vertices $v_{i,\alpha}$ for $1\leq i\leq k$ and $1\leq \alpha\leq \gamma_i$. For $i\in [k]$, the restriction of $G_{Q,\gamma}$ to the vertices $v_{i,1},\dots, v_{i,\gamma_i}$ is the clique on $\gamma_i$ vertices with $a_{ii}-1$ edges between any two distinct vertices. In particular, if $a_{ii}=1$, then we have a collection of $\gamma_i$ totally disconnected vertices.
For $i\neq j\in [k]$ we draw $a_{ij}$ edges between any vertex $v_{i,\alpha_1}$ and $v_{j,\alpha_2}$ for $1\leq \alpha_1\leq \gamma_i$ and $1\leq \alpha_2\leq \gamma_j$.
This determines $G_{Q,\gamma}$. See Figure~\ref{fig:Gqgamma} for  an example of this construction.

\begin{figure}
\includegraphics[scale=0.75]{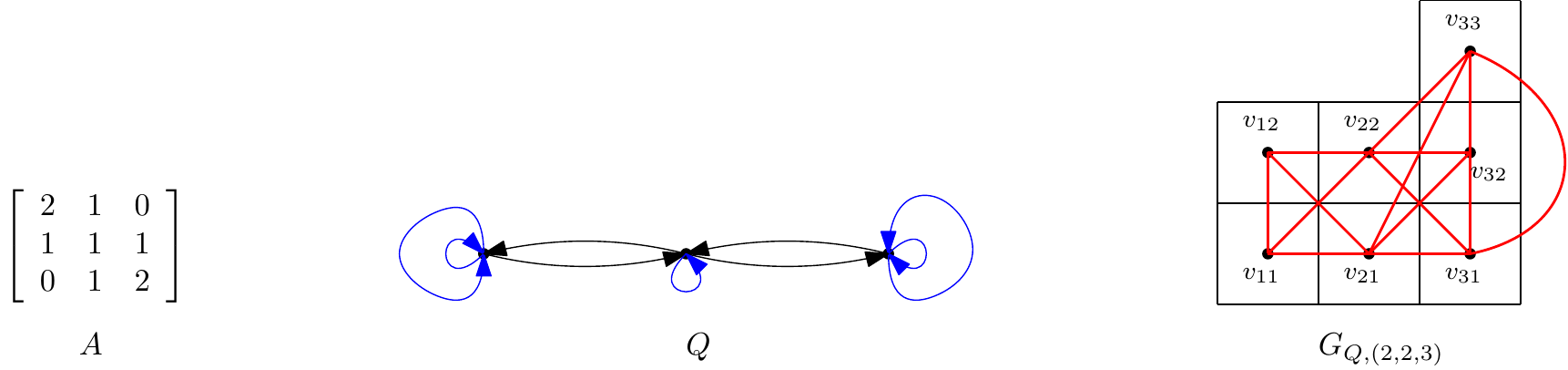}
\caption{A symmetric matrix, its associated quiver, and an instance of $G_{Q,\gamma}$}
\label{fig:Gqgamma}
\end{figure}

\emph{We will assume throughout that $G_{Q,\gamma}$ is connected. Furthermore, for the remainder of this section, we fix our symmetric quiver $Q$ and dimension vector $\gamma$, and we will drop them from notation. In particular, unless otherwise noted, we let $G\coloneqq G_{Q,\gamma}$.
}

%Note that $G_{Q,(2,0,3)}$ is disconnected, which happens because the middle vertex in the quiver is a cut vertex and the associated entry in the dimension vector $\gamma$ is $0$.
%In the second instance of $G_{Q,(2,2,3)}$, this cut vertex has a positive integer corresponding to it in the dimension vector and this ensures that $G_{Q,\gamma}$ is connected.

We reinterpret the important element $f_{\delta,\bar{\delta}}\in J_{\gamma}$ as a polynomial $p_Y$ for a cut in $G$.
\begin{lemma}
\label{lem:efimov_translated}
Consider a decomposition $\gamma=\delta+\bar{\delta}$ where $\delta,\bar{\delta}\neq \gamma$.
The following hold.
\begin{enumerate}
\item $f_{\delta,\bar{\delta}}=p_Y$ for $Y\in \cuts(G)$.
\item For $S\subseteq \vertices(G)$ such that $\partial(S)\in \bonds(G)$, we have that $p_{\partial(S)}\in J_{\gamma}$.
More specifically, there exists a decomposition $\gamma=\delta+\bar{\delta}$ such that $\sigma(f_{\delta,\bar{\delta}})=p_{\partial(S)}$ for some $\sigma\in S_{\gamma}$.
\end{enumerate}
\end{lemma}
\begin{proof}
  Recall that
  \begin{align}
    \label{eq:f_delta_again}
    f_{\delta,\bar{\delta}}=\textcolor{blue}{\displaystyle\prod_{i\neq j\in[k]}\prod_{\alpha_1=1}^{\delta_i}\prod_{\alpha_2=\delta_j+1}^{\gamma_j}(x_{j,\alpha_2}-x_{i,\alpha_1})^{a_{ij}}} \textcolor{magenta}{\displaystyle\prod_{i\in[k]}\prod_{\alpha_1=1}^{\delta_i}\prod_{\alpha_2=\delta_i+1}^{\gamma_i}(x_{i,\alpha_2}-x_{i,\alpha_1})^{a_{ii}-1}}.
  \end{align}
  We understand each triple product separately.

%Throughout this proof $G\coloneqq G_{(Q,\gamma)}$.
Decompose $\vertices(G)$ as $\vertices_1\sqcup \cdots \sqcup \vertices_k$ where
\[
\vertices_i\coloneqq \{v_{i,\alpha}\suchthat 1\leq \alpha\leq \gamma_i\}.
\]
The decomposition $\gamma=\delta+\bar{\delta}$ induces a decomposition of each $\vertices_i=S_i\sqcup \bar{S}_i$ where
\begin{align*}
S_i=\{v_{i,\alpha}\suchthat 1\leq\alpha\leq \delta_i\}, \hspace{10mm}\bar{S}_i=\{v_{i,\alpha}\suchthat \delta_i+1\leq\alpha\leq \gamma_i\}.
\end{align*}
Set $S\coloneqq \sqcup_{1\leq i\leq k}S_i$ and consider the polynomial $p_{\partial(S)}$ attached to the cut $\partial(S)$.
Given the construction of $G$, edges in $\partial(S)$ come in the following two flavors.
\begin{itemize}
\item For $i\neq j\in [k]$, every vertex in $S_i$ is connected to every vertex in $\bar{S}_j$ via $a_{ij}$ edges.
\item For $i\in [k]$, every vertex in $S_i$ is connected to every vertex in $\bar{S}_i$ via $a_{ii}-1$ edges.
\end{itemize}
It follows that $f_{\delta,\bar{\delta}}=p_{\partial(S)}$, which implies the first part of the claim.
Figure~\ref{fig:f_delta_example} decomposes the graph in Figure~\ref{fig:Gqgamma} with $\gamma=(2,2,3)=(1,0,2)+(1,2,1)$. The shaded cells contain the vertices in $S$. Observe that all edges have one end point in the gray shaded region and the other end point in the unshaded region. Edges colored blue (respectively magenta) contribute to the blue (respectively magenta) term in the expression for $f_{\delta,\bar{\delta}}$ in~\eqref{eq:f_delta_again}.
\begin{figure}[!h]
\includegraphics[scale=0.75]{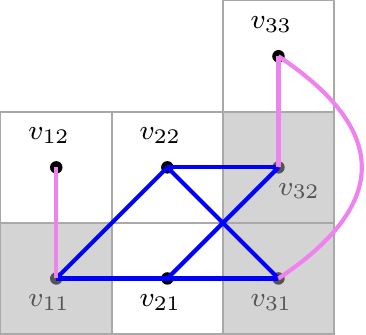}
\caption{$f_{\delta,\bar{\delta}}$ corresponding to a cut}
\label{fig:f_delta_example}
\end{figure}

We  proceed to prove the second statement. Consider a nonempty proper subset $S\subset \vertices(G)$ such that $\partial(S)\in \bonds(G)$.
This induces a decomposition
\begin{align}
S=S_1\sqcup\cdots \sqcup S_k
\end{align}
by letting $S_i\coloneqq S\cap \vertices_i$ for $i\in [k]$.
Define $\delta$ to equal $(|S_1|,\dots,|S_k|)$. This then allows us to decompose $\gamma=\delta+\bar{\delta}$.
Since $S$ is a nonempty proper subset of $\vertices(G)$, we are guaranteed that neither $\delta,\bar{\delta}\neq\gamma$.
We can find a permutation $\sigma\in S_{\gamma}$ that sends $S_i$ to the `initial' set of vertices $\{v_{i,\alpha}\suchthat 1\leq \alpha\leq \delta_i\}$. Since $S_{\gamma}$ is a subgroup of $\aut{G}$, hitting $\vertices(G)$ with $\sigma$  permutes bonds.
Thus we have that $p_{\partial(S)}$, up to reindexing variables, is $f_{\delta,\bar{\delta}}$. Since $J_{\gamma}$ is closed under the $S_{\gamma}$-action, the proof is complete.
\end{proof}
Having inferred that $J_{\gamma}$ is the ideal generated by $p_Y$ for $Y\in \bonds(G)$, by appealing to our discussion at the end of Subsection~\ref{subsec:variations}, we obtain the following corollary that `combinatorializes' Efimov's construction .

\begin{corollary}
\label{cor:J_gamma_revisited}
The following decomposition holds:
\[
A_{\gamma}^{\prim}=\mc{P}(G)\oplus J_{\gamma}.
\]
\end{corollary}
\begin{proof}
Recall that
$A_{\gamma}^{\prim}=\bQ[x_{j,\alpha_2}-x_{i,\alpha_1}\suchthat 1\leq i,j\leq k,1\leq \alpha_1\leq \gamma_i, 1\leq \alpha_2\leq \gamma_j].$
Since $p_Y$ for $Y\in \cuts(G)$ is in the ideal  generated by $p_{Y'}$ for $Y'\in \bonds(G)$, we get from Lemma~\ref{lem:efimov_translated}  that
\begin{align}
J_{\gamma}= \hat{I}(G).
\end{align}
The claim follows by comparison with Equation~\eqref{eq:central_p_decomposition}.
\end{proof}

Now we can realize $\tilde{\Omega}_{\gamma}(q)$ in terms of the space $\mc{P}(G)$ with this new perspective.
Given that usual degree $k$ polynomials end up in the (unusual) grading $2k+\chi_Q(\gamma,\gamma)$ in Efimov's context, by appealing to Corollary~\ref{cor:J_gamma_revisited}, we immediately get our second main result showing that the dimensions of the graded pieces of $\mc{P}(G)^{S_{\gamma}}$ encode coefficients of $\tilde{\Omega}_{\gamma}(q)$.
\begin{theorem}
\label{thm:graded_multiplicity_quantum_DT}
We have the equality
\begin{align*}
\hilb(\mc{P}(G)^{S_{\gamma}})=\tilde{\Omega}_{\gamma}(q).
\end{align*}
\end{theorem}
A couple of remarks are in order.
Note that we have bypassed the intricate argument in the proof of \cite[Theorem 3.10]{Efi} wherein Efimov establishes that only finitely many $V_{\gamma,k}^{\prim}$ are nonzero. This is also where the quantity $N_{\gamma}(Q)$ enters the picture in \emph{loc. cit.}. In fact, by~\eqref{eq:hilb_p(x)} we know that $\hilb(\mc{P}(G))=q^{|\edges(G)|-|\vertices(G)|+1}T_G(1,q^{-1})=q^{g(G)}T_G(1,q^{-1})$.
Thus, we infer the pleasant fact that the degree of $\hilb(\mc{P}(G)^{S_{\gamma}})$, and thereby $\tilde{\Omega}_{\gamma}(q)$, is bounded above by the genus of $G$.
In summary, we have  in fact succinctly established \cite[Theorem 3.10]{Efi} once we show that  the quantity $N_{\gamma}(Q)$ is essentially the genus.
The straightforward proof of the next lemma  is omitted.
\begin{lemma}
\label{lem:N_gamma_reinterpreted}
We have the equality: $g(G)=N_{\gamma}(Q)-1$.
\end{lemma}
% \begin{proof}
% Fix $i\in [k]$ and $1\leq \alpha\leq \gamma_i$.
% The vertex $v_{i,\alpha}$  has the following degree in the graph $G$:
% \[
% (a_{ii}-1)(\gamma_i-1)+\sum_{j\neq i}a_{ij}\gamma_j,
% \]
% where the first term comes from computing degree within the subgraph induced on  $\vertices_i=\{v_{i,s}\suchthat 1\leq s\leq \gamma_i\}$, and the sum on the right comes from considering edges incident to $v_{i,\alpha}$ in the cut $\partial(\vertices_i)$. In particular, the total sum of degrees over all vertices in $G$ is given by
% \[
% \sum_{1\leq i\neq \leq k}a_{ij}\gamma_i\gamma_j+\sum_{1\leq i\leq k}(a_{ii}-1)\gamma_i(\gamma_i-1).
% \]
% Recall that $|\edges(G)|$ is half the sum of degrees over all vertices. Furthermore,  $|\vertices(G)|=\sum_{1\leq i\leq k}\gamma_i$. The claim is now immediate upon comparison with~\eqref{eq:n_gamma}.
% \end{proof}
% +++++++++++++++++++++++++++++

%Observe that Lemma~\ref{lem:N_gamma_reinterpreted} implies that the degree of $\tilde{\Omega}_{\gamma}(q)$ is bounded by the genus $g(G_{Q,\gamma})$.

%%%%%%%%%%%%%%%%%%%%%%%%%%%%%%%%%%%%
\subsection{Numerical DT invariants}
\label{subsec:numerical_DT}
%%%%%%%%%%%%%%%%%%%%%%%%%%%%%%%%%%%%

Define the \emph{numerical} DT-invariant \begin{align}
\mathrm{DT}_{Q,\gamma}\coloneqq \tilde{\Omega}_{\gamma}(1).
\end{align}
We will drop $Q$ from the subscript and write $\mathrm{DT}_{\gamma}$ when it is clear from context.

We are ready to give a manifestly nonnegative combinatorial interpretation to these numbers.
To the best of the authors knowledge, while several explicit formulae \textemdash{} invariably signed because of the presence of the number-theoretic M\"{o}bius function (see next section for such expressions, and also \cite{PSS18})\textemdash{} and other algebraic interpretations are available, ours is the first \emph{combinatorial} interpretation.

We let $n=\sum_{i}\gamma_i$ and then identify $\vertices(G)$ with $[n]$ by relabeling  $v_{1,1},\dots,v_{1,\gamma_1},\dots,v_{k,1},\dots, v_{k,\gamma_k}$ with integers from  $1$ through $n$ in that order. This allows us to identify the bi-indexed variables $x_{1,1},\dots, x_{k,\gamma_k}$ with
$x_1,\dots,x_{n}$.

\begin{theorem}
\label{thm:numerical_dt_break}
The number of $S_{\gamma}$ orbits on $\brkd(G)$ equals $\mathrm{DT}_{\gamma}$.
\end{theorem}
\begin{proof}
Theorem~\ref{thm:piecing_things} gives the isomorphism of graded $\aut{G}$-representations:
\begin{align}
\mc{P}(G)\cong \bQ[\bfx{n}]/\mc{I}(G),
\end{align}
and then says that the right-hand side  is isomorphic to $\bQ[\brkd(G)]$ as an ungraded $\aut{G}$-representation.
Now, $S_{\gamma}$ is a subgroup of $\aut{G}$. By taking $S_{\gamma}$-invariants we get
\begin{align}
\dim(\mc{P}(G)^{S_{\gamma}})=\dim(\bQ[\brkd(G)]^{S_{\gamma}}).
\end{align}
The left-hand side equals $\mathrm{DT}_{\gamma}=\tilde{\Omega}_{\gamma}(1)$, whereas the right-hand side equals the number of $S_{\gamma}$-orbits on $\brkd(G)$. The claim follows.
\end{proof}

\begin{remark}
  \emph{
The preceding proof employs the consequences of the point-orbit method in a crucial way.
Given the succinctness and simplicity of the statement, one naturally wonders if there is an alternative proof.}
\end{remark}

\begin{example}
\emph{
Consider the quiver $Q$ in Figure~\ref{fig:Gqgamma_2}. Pick $\gamma=(2,3)$.
The resulting graph $G$ is the bipartite graph $K_{2,3}$.
We reindex our variables $X_1= x_{11}$, $x_2=x_{12}$, $x_3=x_{21}$, $x_4=x_{22}$, and $x_5=x_{32}$.
We let $S_2\times S_3$ act on $\bQ[x_1,\dots,x_5]$ by letting $S_2$ (respectively $S_3$) act on $x_1,x_2$ (respectively $x_3,x_4,x_5$).
}

\emph{
 $\mc{P}(G)$ is spanned by $p_Y$ for slim subgraphs $Y$ which may be obtained as the $S_{\gamma}$-orbit of elements in $\{1,x_1-x_3, (x_1-x_3)(x_1-x_4), (x_1-x_3)(x_2-x_3) \}$.
The space $\mc{P}(G)^{S_\gamma}$ has basis elements
\[
\{1,\sum_{\sigma\in S_{\gamma}}\sigma\cdot(x_1-x_3), \sum_{\sigma\in S_{\gamma}} \sigma\cdot (x_1-x_3)(x_1-x_4), \sum_{\sigma\in S_{\gamma}}\sigma\cdot (x_1-x_3)(x_2-x_3)).\}.
\]
Explicitly, other than the constant polynomial $1$, these equal
\begin{align*}
&3(x_1+x_2)-2(x_3+x_4+x_5)\\
&3(x_1^2+x_2^2)-2(x_1+x_2)(x_3+x_4+x_5)+2(x_3x_4+x_3x_5+x_4x_5)\\
& 6x_1x_2-2(x_1+x_2)(x_3+x_4+x_5)+2(x_3^2+x_4^2+x_5^2).
\end{align*}
We thus infer that
\[
\tilde{\Omega}_{\gamma}(q)=1+q+2q^2,
\]
and therefore that $\tilde{\Omega}_{\gamma}(1)=4$. Going back to Example~\ref{ex:break_divisors_orbits}, this agrees with the number of $(S_2\times S_3)$-orbits on $\brkd(K_{2,3})$.
}

\begin{figure}
\includegraphics[scale=0.75]{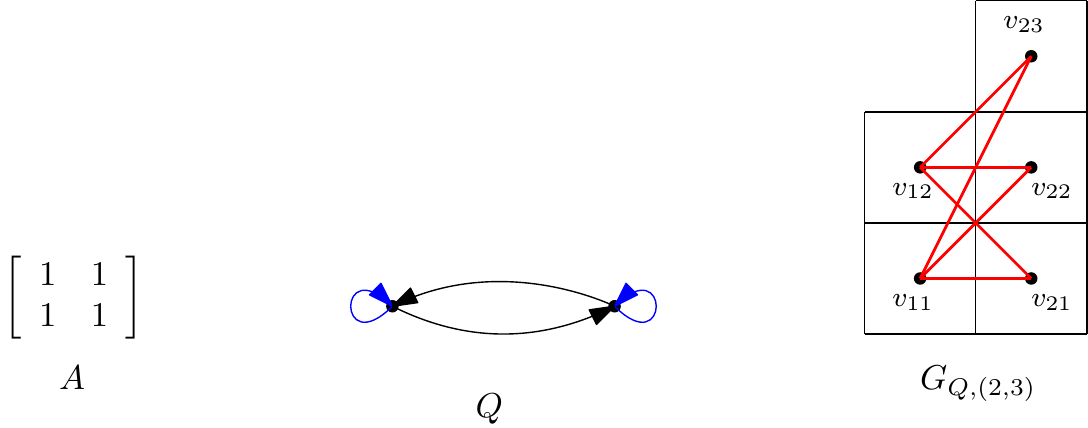}
\caption{A symmetric matrix, its associated quiver, an instance of $G_{Q,\gamma}$.}
\label{fig:Gqgamma_2}
\end{figure}
\end{example}

\begin{remark}
\emph{This example (and others) can be independently verified using a special property of Donaldson-Thomas invariants for quivers with stability; see \cite{ReinekeSmall} and the references therein.}

\emph{These more refined invariants are defined whenever the restriction of the Euler form of a quiver $Q$ to a level set of a stability function $\Theta$ on $Q$ is symmetric \cite[Section 2.5.]{ReinekeSmall}. If the dimension vector $\gamma$ is indivisible, and $\Theta$ is sufficiently generic, the refined DT invariants equals the Poincar\'e polynomial in cohomology of a smooth moduli space of $\Theta$-stable representations of the quiver, which can be computed using a resolved Harder-Narasimhan recursion \cite[Theorem 2.2.]{ReinekeSmall}. If $Q$ is already symmetric, the refined DT invariant equals the ordinary one (a special case of \cite[Corollary 4.4]{ReinekeSmall}), and this allows for its computation without reference to Euler product factorization of a motivic generating series.}
\end{remark}

%%%%%%%%%%%%%%%%%%%%%%%%%%%%%%%%%
\subsection{Linking Efimov and Hausel--Sturmfels}
%%%%%%%%%%%%%%%%%%%%%%%%%%%%%%%
Note that one consequence of Corollary~\ref{cor:J_gamma_revisited}, by invoking~\eqref{eq:hilb_p(x)}, is that
\begin{align}
  \hilb(\mc{P}(G))=q^{g(G)}T_G(1,q^{-1}).
\end{align}
That the Tutte polynomial appears naturally after we have recast Efimov's construction suggests a concrete connection with work of Hausel--Sturmfels \cite{HS02}; see also recent work of Abdelgadir--Mellit--Rodriguez-Villegas \cite{AMV21} where the hint is explicit in the title itself.
With the benefit of hindsight, we arrive at the curious fact that Efimov's construction \cite[Section 3]{Efi} is precisely one of three constructions for the cohomology rings of toric Nakajima quiver varieties given in \cite{HS02}.
We simply recall the pieces that we need and refer the reader to \emph{loc. cit.} for more.

Like in the previous subsection, we work with the identification $\vertices(G)=[n]$.
Orient edges $\{i<j\}\in E(G)$ so that they point from from $j$ to $i$.
The resulting directed graph may itself be treated as a quiver, which we call $\tilde{Q}$.
Attached to this quiver, and by taking the dimension vector $(1^n)$, one obtains a toric Nakajima quiver variety $Y(\tilde{Q})$.
We note that the definition for this variety in \cite{HS02} has an additional parameter $\theta$. While this parameter plays a role in determining the variety, the description of its cohomology ring does not depend on $\theta$.

By unwinding the definitions in \cite[Section 7]{HS02}, it can be checked that the `economical' presentation for the cohomology ring $H^*(Y(\tilde{Q});\bQ)$ is what is discussed in Section~\ref{subsec:variations}. Indeed, the matrix $A$ employed by Hausel--Sturmfels to determine $Y(\tilde{Q})$ is the oriented incidence matrix for $G$, i.e. its columns are obtained by first taking the vectors $e_i-e_j$ for edges $\{i<j\}\in E(G)$ and subsequently striking out the first row.
The quotienting ideal on the right-hand side of \cite[Equation 37]{HS02} that describes $H^*(Y(\tilde{Q});\bQ)$ is, in the language of Section~\ref{sec:central_zonotopal}, equal to the cocircuit ideal $I(A)$ as defined in  equation~\eqref{eq:def_cocircuit_ideal}.
In summary, the quantum DT invariants are determined by the $S_\gamma$-invariant space of the cohomology ring of the toric Nakajima quiver variety $Y(\tilde{Q})$.

While the similarity of this last statement to a result of Hausel--Letellier--Rodriguez-Villegas \cite[Corollary 1.5, Equation 1.10]{HLRV13} is striking, the quiver variety in \emph{loc. cit.} is not a toric Nakajima quiver variety. In particular, our result is not subsumed by results in \cite{HLRV13}.

Taking a cue from \cite{HLRV13} we now discuss some more representation-theoretic aspects of the $S_{\gamma}$-module $\mc{P}(G)$.

%%%%%%%%%%%%%%%%%%%%%%%%%%%%%%%%%%%
\section{Actions have consequences}
\label{sec:applications}
%%%%%%%%%%%%%%%%%%%%%%%%%%%%%%%%%%%

We now focus on representation-theoretic applications of the theory developed in this article, with motivation stemming from work of Berget--Rhoades \cite{BR14} (see also work of Berget \cite{Ber18} in the context of internal zonotopal algebras).
In the interest of brevity, we assume that the reader is conversant with the vocabulary of symmetric functions and associated combinatorics in the context of symmetric group representations, and refer them to \cite{Mac95, St99} for more on this front.
We let $\Lambda$ denote the ring of symmetric functions. We define the \emph{Frobenius characteristic map} $\frob:\oplus_{n\geq 0}\mathrm{Rep}(S_n) \to \Lambda$ by decreeing that  the irreducible representation indexed by the partition $\lambda\vdash n$ is sent to Schur function $s_{\lambda}$, and then extending linearly.
 Here $\mathrm{Rep}(S_n)$ denotes the representation ring of $S_n$.
  If $V=\oplus_{d\geq 0} V_d$ is a graded $S_n$-module, then we define the \emph{graded Frobenius characteristic} $\grfrob(V;q)$
  as $\sum_{d\geq 0}q^d \frob(V).$
Following earlier convention, unless otherwise stated, we take $G$ to be the graph $G_{Q,\gamma}$.

%%%%%%%%%%%%%%%%%%%%%%%%%%%%%%%%%%%%
\subsection{The sign-isotypic component also computes DT invariants}
\label{subsec:sign_isotypic}
%%%%%%%%%%%%%%%%%%%%%%%%%%%%%%%%%%%

We do not currently have a decomposition for  $\mc{P}(G)$ into $S_\gamma$-irreducibles. Having assigned meaning to the $S_\gamma$-invariant space, we shed some light on a close relative thereof.

Recall that $n=|\vertices(G)|$.
Using the \emph{sign} representation $\varepsilon_n$ of  $S_{n}$, we can obtain the sign representation $\varepsilon_{\gamma}$ for $S_{\gamma}$ by restriction.
% Given a tuple $\gamma=(\gamma_1,\dots,\gamma_k)\in \bZ_{\geq 0}^k$, we extend the preceding definition multiplicatively to define the sign representation $\varepsilon_{\gamma}$ of $S_{\gamma}$ as:
% \[
% \varepsilon_\gamma\coloneqq \varepsilon_{\gamma_1}\times \cdots \times \varepsilon_{\gamma_k}.
% \]
Given an $S_{\gamma}$-module $V$, let us denote the sign-isotypic component by $V^{\varepsilon_{\gamma}}$.

Assume like before that $A=(a_{ij})_{i,j\in [k]}$ is a symmetric matrix with nonnegative entries which in turn determines a symmetric quiver $Q$. Let $\gamma=(\gamma_1,\dots,\gamma_k)$ be a dimension vector. Define $Q^+$ to be the symmetric quiver corresponding to the matrix $A+I_k$ where $I_k$ denotes the $k\times k$ identity matrix. Equivalently, $Q^+$ is obtained by adding one new loop at each vertex of $Q$. We let
\[
G^+\coloneqq G_{Q^+,\gamma}.
\]

We have the following result saying that $\varepsilon_{\gamma}$-isotypic components also compute DT invariants.
Indeed, our slim subgraph perspective renders this fact quite transparent.

\begin{proposition}
We have that
\[
\dim(\mc{P}(G^+)^{\varepsilon_{\gamma}})=\mathrm{DT}_{Q,\gamma}.
\]
\end{proposition}
\begin{proof}
%For brevity, we set $G\coloneqq G_{Q,\gamma}$, $G^+=G_{Q^+,\gamma}$, and $\varepsilon\coloneqq \varepsilon_{\gamma}$.
	Note that the right-hand side above is simply $\dim(\mc{P}(G)^{S_{\gamma}})$.
	Thus it suffices to establish an isomorphism between $\mc{P}(G^+)^{\varepsilon_{\gamma}}$ and $\mc{P}(G)^{S_{\gamma}}$. Consider the map sending
	\begin{align}
	\label{eq:vandermonde_multiplication}
	f\in \mc{P}(G) \mapsto f^+\coloneqq f \cdot \prod_{1\leq i\leq k}\prod_{1\leq \alpha_1<\alpha_2\leq \gamma_i}(x_{i,\alpha_1}-x_{i,\alpha_2})
	\end{align}
	Observe that the innermost product is  the Vandermonde determinant on the variables $x_{i,1}$ through $x_{i,\gamma_i}$. Each linear form $(x_{i,\alpha_1}-x_{i,\alpha_2})$ corresponds to an edge in $E(G^+)\setminus E(G)$.

	First we claim that $f^+\in \mc{P}(G^+)$. It suffices to verify this for a slim $Y\subset E(G)$. The polynomial $p_Y^+$ is attached to the set of edges $Y^+\coloneqq \left(E(G^+)\setminus E(G)\right) \sqcup Y$. This and the fact that $Y$ is slim together imply that $Y^+$ is slim in $G^+$. Thus $f\in \mc{P}(G^+)$ indeed.

	Now the fact that the association in \ref{eq:vandermonde_multiplication} is an isomorphism between $\mc{P}(G)^{S_{\gamma}}$ and $\mc{P}(G^+)^{\varepsilon_{\gamma}}$ is straightforward.
	Indeed any polynomial $f$ in the variables $(x_{i,\alpha})_{1\leq i\leq k,1\leq \alpha\leq \gamma_i}$ with the property that $\sigma \cdot f=\varepsilon(\sigma)$ for $\sigma \in S_{\gamma}$ must be divisible by $\prod_{1\leq i\leq k}\prod_{1\leq \alpha_1<\alpha_2\leq \gamma_i}(x_{i,\alpha_1}-x_{i,\alpha_2})$ with the ratio being $S_{\gamma}$-invariant.
	This concludes the proof.
\end{proof}

%%%%%%%%%%%%%%%%%%%%%%%%%%%%%%
\subsection{The top degree of $\mc{P}(G)$ as an $S_\gamma$-module}
\label{subsec:top_degree}
%%%%%%%%%%%%%%%%%%%%%%%%%%%%%%

Our second result generalizes a result of Berget--Rhoades \cite[Theorem 5]{BR14} which identifies the top degree graded piece of $\mc{P}(K_n^m)$ with $\mathrm{Lie}_n\otimes \varepsilon_n$.
Here $\mathrm{Lie}_n$ is the well-known \emph{Lie representation} of $S_n$ with the property that upon restriction to $S_{n-1}$, one recovers the regular representation. In particular, the dimension of $\mathrm{Lie}_n$ is $(n-1)!$.
One description of $\mathrm{Lie}_n$ is as the action on $S_n$ on the multilinear part of the free Lie algebra on $n$ symbols.

The graph $G$ (with $\vertices(G)$ identified with $[n]$ like before) determines a hyperplane arrangement $\mc{A}_G$ in $\bC^n$ by considering the union of hyperplanes $x_i-x_j=0$ for edges $\{i,j\}\in E(G)$.
We consider the \emph{de Rham cohomology ring} $H^*(\bC^n\setminus \mc{A}_G)$ over $\bC$ of the complement $\bC^n\setminus \mc{A}_G$.
Thanks to seminal work of Orlik--Solomon \cite{OS80,OT92}, $H^*(\bC^n\setminus \mc{A}_G)$ is very well understood and it is the exterior algebra over the generators $\mathrm{d}\log(\beta)=\frac{d\beta}{\beta}$ where $\beta$ ranges over the linear forms that cut out $\mc{A}_G$, which in our case are the $x_i-x_j$ for edges $\{i,j\}\in E(G)$.
Recall our collection of vectors $\hat{X}_G$ introduced in Section~\ref{subsec:variations}, and endow this collection with a total order.
The top degree graded piece of $H^*(\bC^n\setminus \mc{A}_G)$, denoted by $H^*(\bC^n\setminus \mc{A}_G)_{\mathrm{top}}$, has a basis indexed by bases in $\hat{X}_G$ with external activity $0$, i.e. unbroken bases. Such bases also form a basis for the top degree of $\mc{P}(G)$, denoted by $\mc{P}(G)_{\mathrm{top}}$. This is easily seen from the expression~\eqref{eq:hilb_p(x)} describing $\hilb(\mc{P}(G))$ as a specialization of the Tutte polynomial of $\hat{X}_G$.
Thus we see that
\begin{align}
\label{eq:equality_of_dimensions}
\dim(\mc{P}(G)_{\mathrm{top}})=\dim(H^*(\bC^n\setminus \mc{A}_G)_{\mathrm{top}}).
\end{align}

\begin{remark}
The quantity in~\eqref{eq:equality_of_dimensions} is known to equal the number of bounded regions in the complement $\bR^n\setminus \mc{A}_G$ of the \emph{real} arrangement.
\end{remark}

In fact, following Berget--Rhoades \cite{BR14}, we can say more.
\begin{proposition}
\label{prop:top_degree}
We have the following $S_{\gamma}$-isomorphism:
\[
\mc{P}(G)_{\mathrm{top}}\cong H^*(\bC^n\setminus \mc{A}_G)_{\mathrm{top}}\otimes \varepsilon_{\gamma_1}^{a_{11}-1}\otimes \cdots \otimes \varepsilon_{\gamma_k}^{a_{kk}-1}.
\]
Recall that $a_{ii}-1$ is one less than the number of loops at vertex $i$ in $Q$.
\end{proposition}

\begin{proof}
Note that $\mc{P}(G)_{\mathrm{top}}$ is the span of all $p_Y$ for $Y\subset E(G)$, where $E(G)\setminus Y$ is the edge set of a spanning tree in $G$. For a fixed such $Y$, consider the map
\begin{align}
\label{eq:association_top_degree}
p_Y\mapsto p_Y\cdot (\mathrm{d}(x_1-x_2)\wedge \cdots \wedge \mathrm{d}(x_{n-1}-x_n))/\prod_{\{i,j\}\in E(G)}(x_i-x_j).
\end{align}
We claim that the right-hand side is indeed in $H^*(\bC^n\setminus \mc{A}_G)$. Note that $\{x_1-x_2,\dots,x_{n-1}-x_n\}$ are a basis for an $(n-1)$-dimensional space;\footnote{The reflection representation of $S_n$} let us call it $V$. Thus $\bigwedge^{n-1}V$ is one dimensional. If $\{i_1,j_1\}, \dots, \{i_{n-1},j_{n-1}\}$ are the edges of any spanning tree of $G$, then $\{x_{i_1}-x_{j_1},\dots,x_{i_{n-1}}-x_{j_{n-1}}\}$ is also a basis for $V$. Thus we must have that
\begin{align}
\mathrm{d}(x_1-x_2)\wedge \cdots \wedge \mathrm{d}(x_{n-1}-x_n)=\pm \mathrm{d}(x_{i_1}-x_{j_1})\wedge \cdots \wedge \mathrm{d}(x_{i_{n-1}}-x_{j_{n-1}}).
\end{align}
Thus the right-hand side in~\eqref{eq:association_top_degree} is, up to a sign, equal to
\[
\mathrm{d}\log(x_{i_1}-x_{j_1}) \wedge \cdots \wedge \mathrm{d}\log(x_{i_{n-1}}-x_{j_{n-1}}),
\]
which is indeed an element of $H^*(\bC^n\setminus \mc{A}_G)_{\mathrm{top}}$. It also follows that the map is surjective. The equality of dimensions in~\eqref{eq:equality_of_dimensions} implies that we have a vector space isomorphism.

We proceed to verify that it is in fact an $S_{\gamma}$-module isomorphism.
Now pick $\sigma\coloneqq \sigma^{(1)}\cdots \sigma^{(k)}\in S_{\gamma_1}\times \cdots\times S_{\gamma_k}$.
We know that $\bigwedge^{n-1}(V)$ carries the sign representation of $S_n$, and thus, on the one hand we get
\begin{align}
\sigma\cdot	(\mathrm{d}(x_1-x_2)\wedge \cdots \wedge \mathrm{d}(x_{n-1}-x_n))=(\mathrm{d}(x_1-x_2)\wedge \cdots \wedge \mathrm{d}(x_{n-1}-x_n))\prod_{1\leq i\leq k}\varepsilon_{\gamma_i}(\sigma^{(i)}).
\end{align}
On the other hand we get
\begin{align}
\sigma \cdot \prod_{\{i,j\}\in E(G)}(x_i-x_j)=\prod_{\{i,j\}\in E(G)}(x_i-x_j)\prod_{1\leq i\leq k}(\varepsilon_{\gamma_i}(\sigma^{(i)}))^{a_{ii}}.
\end{align}
Thus the action of $\sigma$ on $\mc{P}(G)$ is the action on $H^*(\bC^n\setminus \mc{A}_G)_{\mathrm{top}}$ twisted precisely by the appropriate powers of   sign representations.
\end{proof}

\begin{remark}
Note that if all $a_{ii}$ are odd positive integers, i.e. we have an odd number of loops at each vertex, then we get the isomorphism $\mc{P}(G)_{\mathrm{top}}\cong H^*(\bC^n\setminus \mc{A}_G)_{\mathrm{top}}$ holds.
\end{remark}

%%%%%%%%%%%%%%%%%%%%%%%%
\subsection{The $(m+1)$-loop quiver}
\label{subsec:loop quiver}
%%%%%%%%%%%%%%%%%%%%%%%%

In the notation of this article, Berget--Rhoades \cite{BR14} consider the $S_n$-modules $\mc{P}(K_n^m)$ and establish several interesting results.
While the graded Frobenius characteristics of these modules  remains elusive, we are able to shed light on related matters. In particular, we show that this module encodes a large class of DT invariants.
We begin with the case of the $(m+1)$-loop quiver.

Let $m$ be a positive integer.
We begin by revisiting the case of the $(m+1)$-loop quiver $Q$ with dimension vector $\gamma=(n)$ where $n\geq 1$.
In this case, $G_{Q,\gamma}$ equals the complete multipartite graph $K_n^{m}$, i.e. the graph on $n$ vertices with $m$ edges between any two distinct vertices.
In this case $\aut{G}$ is clearly the symmetric group $S_n$.
Theorem~\ref{thm:piecing_things} then gives us (isomorphic) spaces with dimension $m^{n-1}n^{n-2}$.

In view of Theorem~\ref{thm:piecing_things} we can now connect the space $\mc{P}(K_n^{m})$ to $\bQ[\brkd(K_n^m)]$ and, among other things, resolve \cite[Conjecture 3.3]{KT21} and recover one of the main results in \cite{BR14}.
\begin{corollary}
\label{cor:p-space-corollary}
The following hold.
\begin{enumerate}
\item We have the isomorphism $\mc{P}(K_n^m)\cong_{S_n} \bQ[\brkd(K_n^m)]$ of ungraded representations. In particular, the $\frob(\mc{P}(K_n^m))$ is $h$-positive.
\item The graded multiplicity of the trivial representation of $S_n$ in $\mc{P}(K_n^m)$ is given by the quantum DT invariant $\tilde{\Omega}_{(n)}(q)$.
At $q=1$ we obtain
\begin{align}
\label{eq:explicit_m_loop_quiver}
\dim(\mc{P}(K_n^m)^{S_n})=\mathrm{DT}_{Q,(n)}=
\frac{1}{mn^2}\sum_{d|n}(-1)^{mn+\frac{mn}{d}}\mu(d)\binom{\frac{(m+1)n}{d}-1}{\frac{n}{d}}.
\end{align}
\item We have $\grfrob(\mc{P}(K_n)\downarrow_{S_{n-1}}^{S_n})=\grfrob(\mathrm{PF}_{n-1})$, where $\mathrm{PF}_{n-1}$ is the parking function representation of $S_{n-1}$, with basis indexed by parking functions and grading given by the sum of entries in a parking function.
\end{enumerate}
\end{corollary}
\begin{proof}
The first statement is simply Theorem~\ref{thm:piecing_things}. The $h$-positivity is a consequence of the permutation action of $S_n$ on break divisors on $K_n^m$, which are lattice points in the permutahedron determined as the convex hull of the $S_n$-orbit of $(m(n-1)-1,\dots,m\cdot 1-1,0)$. Clearly, the stabilizer of any point is a Young subgroup, from which the claim follows.

The second statement is a consequence of Theorem~\ref{thm:graded_multiplicity_quantum_DT} for $G=K_{n}^m$. At $q=1$, we obtain the claimed expression because it equals the number of orbits under the $S_n$ action on $\brkd(K_n^m)$ \cite[Theorem 3.7]{KRT21}. The explicit expression for this specific numerical DT invariant was first computed by Reineke \cite[Theorem 3.2]{Rei12}.

The final statement was established in \cite[Thm. 2]{BR14}. We give an alternative proof.
We have the graded $S_n$-isomorphism $\bQ[\bfx{n}]/\mc{I}(K_n)\cong \mc{P}(K_n)$.
By \cite{PS03} the left-hand side has a monomial basis $\{x_1^{a_1}\cdots x_{n-1}^{a_{n-1}}\}$ indexed by parking functions $(a_1,\dots,a_{n-1})$ of length
$n-1$.
Although the $S_n$-action on this basis is not easy to describe,
the $S_{n-1}$-action (permuting variables $x_1$ through $x_{n-1}$) is easily seen to be permutation action on parking functions.
Since the degree of a basis element is preserved, we get the graded isomorphism upon restriction as claimed.
\end{proof}

Corollary~\ref{cor:p-space-corollary} (3) may also be stated using symmetric function operators. Let $\langle -, - \rangle$ be the {\em Hall inner product}
on $\Lambda$ which declares the Schur basis to be orthogonal and let $s_1^{\perp}: \Lambda \rightarrow \Lambda$ be the degree $-1$ operator
which is adjoint to multiplication by $s_1$ under this inner product.
Corollary~\ref{cor:p-space-corollary} (3) is equivalent to
\begin{equation}
s_1^{\perp} \mathrm{grFrob}(\mc{P}(K_n);q) = (\mathrm{rev}_q \circ \omega)  \nabla e_{n-1} \mid_{t \rightarrow 1}
\end{equation}
where
\begin{itemize}
\item
$e_{n-1}$ is the degree $n-1$ elementary symmetric function,
\item
$\mathrm{rev}_q$ reverses coefficient sequences of polynomials in $q$,
\item $\omega: \Lambda \rightarrow \Lambda$ is the involution interchanging $s_{\lambda}$ and $s_{\lambda'}$, and
\item  $\nabla: \Lambda \rightarrow \Lambda$ is the eigenoperator on the basis $\{ \widetilde{H}_{\lambda}[X;q,t] \}$ of modified Macdonald polynomials
characterized by $\nabla: \widetilde{H}_{\lambda}[X;q,t] \mapsto q^{\sum (i-1) \cdot \lambda_i} t^{\sum (i-1) \lambda_i'} \cdot \widetilde{H}_{\lambda}[X;q,t]$.
\end{itemize}
Combinatorially, the grading on $s_1^{\perp} \mathrm{grFrob}(\mc{P}(K_n);q)$ corresponds to the coarea statistic when we think of parking
functions in terms of Dyck paths.

\begin{example}
\label{ex:motivation_for_more_DTs}
\emph{
Let us take $m=3$ and $n=4$. Then the ungraded $S_4$ module $\mc{P}(K_4^3)$ has Frobenius characteristic given by
\[
10h_{(1, 1, 1, 1)} + 15h_{(2, 1, 1)} + 3h_{(3, 1)}.
\]
This translates to the following Schur expansion:
\[
10s_{(1,1,1,1)} + 45s_{(2,1,1)} + 35s_{(2,2)} + 63s_{(3,1)} + 28s_{(4)}.
\]
From the coefficient of $s_{(4)}$, we infer that for the $4$-loop single vertex quiver $Q$ with dimension vector $\gamma=(4)$, we have
\[
\mathrm{DT}_{Q,\gamma}=28.
\]
}
\end{example}

As mentioned earlier, we do not yet know a combinatorial description for the Schur expansion of $\grfrob(\mc{P}(K_n^m))$. Even for the Schur expansion of $\frob(\mc{P}(K_n^m))$, one could derive an unconvincing expression involving Kostka numbers by appealing to the $h$-expansion.

This having said, the expansion for $\frob(\mc{P}(K_n^m))$ in the basis of monomial symmetric functions carries interesting information\textemdash{} if
$
\frob(\mc{P}(K_n^m))=\sum_{\lambda\vdash n}c_{\lambda}m_{\lambda},
$
then
\[
c_{\lambda}=\dim(\mc{P}(K_n^m)^{S_{\lambda}}).
\]
This given, one may wonder if there are quivers $Q$ for which $\mathrm{DT}_{Q,\lambda}$ equals $c_{\lambda}$. If so, then one may further ask for explicit expressions in the vein of ~\eqref{eq:explicit_m_loop_quiver}.
We answer these questions next.

%%%%%%%%%%%%%%%%%%%%%%%%%%
\subsection{$\mathrm{DT}_{Q,\lambda}$ for quivers $Q$ and dimension vectors $\lambda$ so that $G_{Q,\lambda}=K_n^m$}
%%%%%%%%%%%%%%%%%%%%%%%%%
Fix positive integers $m$ and $k$. Fix a partition $\lambda=(\lambda_1\geq \cdots\geq \lambda_k>0)$.
Let $A$ be the $k\times k$ symmetric matrix with $a_{ii}=m+1$ for all $1\leq i\leq k$, and $a_{ij}=m$ for $1\leq i\neq j\leq k$.
Then $A$ determines a symmetric quiver $Q$ on $k$ vertices\footnote{This quiver is a subset of the class of \emph{almost $m$-regular quivers} \cite[Definition 6.3]{DFR21}.} with the property that the associated covering graph $G_{Q,\lambda}$ is $K_n^m$. Here $n=|\lambda|\coloneqq \sum_{1\leq i\leq k} \lambda_i$.

Our aim is to give explicit formulae for $\mathrm{DT}_{Q,\lambda}$ that generalize Reineke's formula in the $(m+1)$-loop quiver case.
By Theorem~\ref{thm:numerical_dt_break} we know that $\mathrm{DT}_{Q,\lambda}$ is the number of $S_{\lambda}$-orbits on $\brkd(K_{n}^m)$.
The polyhedral description of break divisors (as lattice points in trimmed permutahedra) does not lend itself well to a combinatorial analysis. To bypass this, we appeal to an alternative characterization for $\brkd(K_n^m)$ from \cite{KRT21} (building upon \cite{KST21,KT21}).

For $g\coloneqq g(K_n^m)$ consider
\begin{align}
\mc{D}_{m,n}\coloneqq \{(y_1,\dots,y_{n})\in \bZ_{\geq 0}^n\suchthat y_1+\cdots+y_n=g(\!\!\!\!\!\mod mn), 0\leq y_i\leq mn-1\}.
\end{align}
Note that $\mc{D}_{m,n}$ carries a commuting $S_n \times \mathbb{Z}/n \mathbb{Z}$-action where the $S_n$-action is the natural one and the $\mathbb{Z}/n \mathbb{Z}$-action is given by adding $m$ modulo $mn$ to all coordinates.
As shown in \cite{KRT21}, break divisors of $K_n^m$ are distinguished representatives under the $\bZ/n\bZ$ action, and the number of $S_n$-orbits on $\brkd(K_n^m)$ can thus be computed by counting $S_n\times \bZ/n\bZ$ orbits on $\mc{D}_{n,m}$. This is precisely the strategy followed in \cite{KRT21}.

To compute $\mathrm{DT}_{Q,\lambda}$, we thus need to compute $S_{\lambda}\times \bZ/n\bZ$ orbits on $\mc{D}_{m,n}$.
We begin by recording some handy elementary number-theoretic facts that will play a crucial role in our analysis.
Let $\mu$ and $\phi$ denote the M\"{o}bius and Euler totient functions respectively.
Let $C_d(b)$ denote the Ramanujan sum as in \cite[\S 3]{KRT21}:
\begin{align}
C_d(b)\coloneqq \sum_{\substack{1\leq k\leq d\\ \gcd(k,d)=1}}e^{2\pi ikb/d}=\mu\left(\frac{d}{\gcd(b,d)}\right)\frac{\phi(d)}{\phi\left(\frac{d}{\gcd(b,d)}\right)}.
\end{align}
The following `orthogonality' result of Cohen \cite[Equation 1.2]{Coh59} (attributed to Carmichael) is pertinent.
For the reader comparing our approach here to that in \cite{KRT21}, we stress the fact that we did not need this property (and its consequences) in arriving at the results in \emph{loc. cit.}.

\begin{proposition}
\label{prop:cohen}
Fix positive integers $p$ and $q$.
Consider divisors $d$ and $e$ of $q$.
Then
\begin{align*}
\sum_{\substack{a+b=p(\!\!\!\!\! \mod q)\\ 0\leq a,b\leq q-1}}C_d(a)C_e(b)=\left\lbrace \begin{array}{ll}0 & d\neq e,\\ qC_d(p) & d=e.\end{array}\right.
\end{align*}
\end{proposition}
As an immediate corollary we obtain the following result that allows for substantial simplification of expressions involving Ramanujan sums.
\begin{corollary}
\label{cor:general_cohen}
Fix positive integers $p$ and $q$.
Consider a tuple $(d_1,\dots,d_k)$ of positive integers such that $d_i|q$ for all $1\leq i\leq k$.
Then
\begin{align*}
\sum_{\substack{a_1+\cdots + a_k=p(\!\!\!\!\! \mod q)\\ 0\leq a_i\leq q-1}}\prod_{1\leq i\leq k}C_{d_i}(a_i)=\left\lbrace \begin{array}{ll}q^{k-1}C_d(p) & d\coloneqq d_1=\cdots=d_k,\\0 & \text{otherwise.}\end{array}\right.
\end{align*}
\end{corollary}
\begin{proof}
The case $k=1$ is vacuously true. The case $k=2$ is Proposition~\ref{prop:cohen}. The remaining argument is a straightforward induction.
\end{proof}

We consider a slightly more general setup.
For $0\leq s\leq mn-1$, define
\begin{align}
\label{eq:general_dmnrs}
\mc{D}_{m,n,r,s}\coloneqq \{(y_1,\dots,y_{r})\suchthat y_1+\cdots+y_r=s(\!\!\!\!\!\mod mn), 0\leq y_i\leq mn-1\},
\end{align}
and let $
\mc{D}_{m,n,s}\coloneqq \mc{D}_{m,n,n,s}.$
Irrespective of $s$, the Young subgroup $S_{\lambda}$ acts on $\mc{D}_{m,n,s}$.
Denote the number of orbits under this action by $O_{m,\lambda,s}$.
We give a formula for $O_{m,\lambda,s}$ using Ramanujan sums.

\begin{proposition}
\label{prop:d_orbits}
Let $\gcd(\lambda)\coloneqq \gcd(\lambda_1,\dots,\lambda_k)$.
We have that
\[
O_{m,\lambda,s}=\frac{1}{mn}\sum_{d|\gcd(\lambda)}C_d(s)\prod_{1\leq i\leq k}\binom{\frac{mn+\lambda_i}{d_i}-1}{\frac{\lambda_i}{d_i}}.
\]
\end{proposition}
\begin{proof}
One may decompose $(y_1,\dots,y_n)\in \mc{D}_{m,n,s}$ as a concatenation of $k$ sequences: $(y_1,\dots,y_{\lambda_1})$, $(y_{\lambda_1+1},\dots,y_{\lambda_1+\lambda_2}),\dots$, $(y_{\lambda_1+\cdots+\lambda_{k-1}+1},\dots,y_{n})$.
Thus, abusing notation slightly, we may interpret $\mc{D}_{m,n,s}$ as being given by the disjoint decomposition:
\begin{align}
\label{eq:decompose_lattice_points}
\mc{D}_{m,n,s}=\displaystyle\bigsqcup_{\substack{a_1+\cdots+a_k=s(\!\!\!\!\!\mod mn)\\ 0\leq a_i\leq mn-1}}\displaystyle\bigsqcup_{1\leq i\leq k}\mc{D}_{m,n,\lambda_i,a_i}.
\end{align}
For any choice of $(a_1,\dots,a_k)$, the symmetric group $S_{\lambda_i}$ acts on the $\mc{D}_{m,n,\lambda_i,a_i}$.
The orbits under this $S_{\lambda_i}$-action on $\mc{D}_{m,n,\lambda_i,a_i}$ are indexed by multisets of cardinality $\lambda_i$ with elements drawn from $\{0,\dots,mn-1\}$ such that the sum of the elements in $a_i$ modulo $mn$. We know \cite[Lemma 3.1]{KRT21}  that this quantity equals
\[
\frac{1}{mn}\sum_{d_i|\gcd(\lambda_i,mn)} \binom{\frac{mn+\lambda_i}{d_i}-1}{\frac{\lambda_i}{d_i}}C_{d_i}(a_i).
\]
From the decomposition in~\eqref{eq:decompose_lattice_points} it then follows that
\begin{align}
O_{m,\lambda,s}=\frac{1}{m^kn^{k}}\sum_{\substack{a_1+\dots+a_k=s(\!\!\!\!\!\mod mn)\\ 0\leq a_i\leq mn-1}}\sum_{\substack{(d_1,\dots,d_k)\\ d_i|\gcd(\lambda_i,mn)}}\prod_{1\leq i\leq k}\binom{\frac{mn+\lambda_i}{d_i}-1}{\frac{\lambda_i}{d_i}}C_{d_i}(a_i),
\end{align}
which by changing the order of summation may be rewritten as
\begin{align}
O_{m,\lambda,s}=\frac{1}{m^kn^{k}}\sum_{\substack{(d_1,\dots,d_k)\\ d_i|\gcd(\lambda_i,mn)}}\prod_{1\leq i\leq k}\binom{\frac{mn+\lambda_i}{d_i}-1}{\frac{\lambda_i}{d_i}}\sum_{\substack{a_1+\dots+a_k=s(\!\!\!\!\!\mod mn)\\ 0\leq a_i\leq mn-1}}\prod_{1\leq i\leq k}C_{d_i}(a_i).
\end{align}
\noindent By Corollary~\ref{cor:general_cohen} the inner summand simplifies and we get the desired expression.
\end{proof}

We record the result that we care about.

\begin{corollary}
\label{cor:more_DTs}
Let $\mc{D}_{m,n}\coloneqq \mc{D}_{m,n,g(K_n^m)}$.
Then we have the equality
\[
\mathrm{DT}_{Q,\lambda}=\frac{1}{mn^2}\sum_{d|\gcd(\lambda)}(-1)^{mn+\frac{mn}{d}}\mu(d)\prod_{1\leq i\leq k}\binom{\frac{mn+\lambda_i}{d}-1}{\frac{\lambda_i}{d}}.
\]
\end{corollary}
\begin{proof}
As argued earlier, the number of orbits under $S_{\lambda}$-action on $\brkd(K_n^m)$ equals $O_{m,\lambda,g}/n$. 
By Proposition~\eqref{prop:d_orbits} this equals
\begin{align}
\label{eq:more general DT}
\frac{1}{mn^2}\sum_{d|\mathrm{gcd}(\lambda)}(-1)^{mn+\frac{mn}{d}}\mu(d)\prod_{1\leq i\leq k}\binom{\frac{mn+\lambda_i}{d}-1}{\frac{\lambda_i}{d}}.
\end{align}
Here we have used the equality $C_d(g)=(-1)^{mn+\frac{mn}{d}}\mu(d)$,  implicit in \cite[Section 3]{KRT21}.
\end{proof}
Observe that the integrality of $\mathrm{DT}_{Q,\lambda}$ from the aforementioned expression is not obvious.
We can also straightaway derive a host of equalities that may be of independent interest.
\begin{enumerate}
\item If $\lambda=(1^n)$, then the quantity in \eqref{eq:more general DT} equals $m^{n-1}n^{n-2}$, as it should. Indeed this is the number of spanning trees in $K_n^m$.
\item If $\lambda=(n)$, then the quantity in \eqref{eq:more general DT} equals the numerical DT-invariant of the $(m+1)$-loop quiver as in \eqref{eq:explicit_m_loop_quiver}.
\item If $\lambda=(n-1,1)$ for $n\geq 2$, then we are considering the $S_{n-1}\times S_1$ action on $\brkd(K_n^m)$.
The quantity in \eqref{eq:more general DT} only sees a contribution from $d=1$, and thus equals
\begin{align*}
\frac{1}{mn^2}\binom{mn+n-2}{n-1}\binom{mn+1-1}{1}=\frac{1}{n}\binom{mn+n-2}{n-1}
\end{align*}
When $m=1$, this becomes the $(n-1)$th Catalan number equaling $\frac{1}{n}\binom{2n-2}{n-1}$.
This is not surprising.
Indeed, by Berget--Rhoades \cite{BR14}, the restriction of the $S_n$ action on break divisors on $K_n$ to $S_{n-1}\times S_1$ recovers the classical parking function representation, whose orbit count is well known to be given by the Catalan numbers.
\item More generally, if $\lambda=(\lambda_1,\dots,\lambda_k)\vdash n$ is such that $\gcd(\lambda)=1$, then
\eqref{eq:more general DT} gives a product formula:
\[
\frac{1}{mn^2}\prod_{1\leq i\leq k}\binom{mn+\lambda_i-1}{\lambda_i}.
\]
\end{enumerate}

\begin{example}
\emph{
We return to Example~\ref{ex:motivation_for_more_DTs} to complete that link. The monomial symmetric function expansion for $\frob(\mc{P}(K_4^3))$ equals
\[
432m_{(1,1,1,1)} + 234m_{(2,1,1)} + 126m_{(2,2)} + 91m_{(3,1)} + 28m_{(4)}.
\]
Let us verify that the expression in Corollary~\ref{cor:more_DTs} does indeed match for $\lambda={(2,2})$. We get
\[
\left(\binom{12+2-1}{2}^2-\binom{7-1}{1}^2\right)/48=126,
\]
which agrees with the coefficient of $m_{(2,2)}$ in $\frob(\mc{P}(K_4^3))$.
}
\end{example}

\section*{Acknowledgements}
V.T. is extremely grateful to Matja\v{z} Konvalinka for earlier collaboration.
Thanks also to Spencer Backman, Chris Eur, Philippe Nadeau, Marino Romero, and Chi Ho Yuen for enlightening conversations/correspondence that have influenced, directly or indirectly, the results obtained in this work.

\bibliographystyle{hplain}
\bibliography{Biblio_PS}

\end{document}